\documentclass[leqno,12pt]{amsart}

\usepackage{amssymb,amsmath,multirow,graphics,multicol}

\theoremstyle{plain}
\newtheorem{theorem}{Theorem}[section]

\numberwithin{equation}{section}

\theoremstyle{remark}
\newtheorem{remark}{Remark}[section]
 \numberwithin{equation}{section}

\def\<{\left < }
\def\>{\right >}
\def\({\left ( }
\def\){\right )}
\def\e{\eqref}

\def\n{\nabla}

\def\e{\eqref}
\def\p{\partial }
\def\la{\lambda}
\def\sa{\sigma}

\newtheorem{thm}{Theorem}[section]
\newtheorem{lem}{Lemma}[section]
\newtheorem{rem}{Remark}[section]
\newtheorem{prop}{Proposition}[section]
\newtheorem{cor}{Corollary}[section]
\newtheorem{defn}{Definition}[section]

\date{}
\input { amssym.def}
\input { amssym.def}
\newcommand{\bg} {\begin{equation}}
\newcommand{\ede} {\end{equation}}

\newcommand{\D}{\Delta}

\newcommand{\vp}{\varphi}

\def\ric{\operatorname{Ric}}
\def\ov{\overline}
\def\div{\operatorname{div}}

\begin{document}
\title[$\Phi$-Harmonic Maps and $\Phi$-Superstrongly Unstable Manifolds]
{$\Phi$-Harmonic Maps \\ and $\Phi$-Superstrongly Unstable Manifolds}
\author[Yingbo Han]{Yingbo Han$^{\ast}$}
\address{School of Mathematics and Statistics\\
    Xinyang Normal University \\Xinyang, 464000, Henan, P. R. China}
\email{yingbohan@163.com}

\author[Shihshu Walter Wei]{Shihshu Walter Wei$^{\ast \ast}$}
\address{Department of Mathematics\\
University of Oklahoma\\ Norman, Oklahoma 73019-0315\\ U.S.A.}
\email{wwei@ou.edu}

\thanks{
$^{\ast}$ Research supported in part by the National Natural Science Foundation of China (Grant No. 11971415, 11701494) and the Nanhu Scholars Program for Young Scholars of XYNU and the Universities Young Teachers Program of Henan Province (2016GGJS-096) and Teacher Education Project of XYNU (2019-JSJYYJ-12).\\
$^{\ast \ast}$ Research supported in part by NSF (DMS-1447008)}


\begin{abstract}In this paper, we motivate and define $\Phi$-energy density, $\Phi$-energy, $\Phi$-harmonic maps and stable $\Phi$-harmonic maps. Whereas harmonic maps or $p$-harmonic maps can be viewed as critical points of the integral of  $\sigma_1$ of a pull-back tensor, $\Phi$-harmonic maps can be viewed as critical points of the integral of $\sigma_2$ of a pull-back tensor.
By an extrinsic average variational method in the calculus of variations (cf. \cite{HW,WY,13,HaW}), we derive the average second variation formulas for $\Phi$-energy functional, express them in orthogonal notation in terms of the differential matrix, and find $\Phi$-superstrongly unstable $(\Phi$-$\text{SSU})$ manifolds. We prove, in particular that every compact $\Phi$-$\text{SSU}$ manifold must be $\Phi$-strongly unstable $(\Phi$-$\text{SU})$, i.e., $(a)$ A compact $\Phi$-$\text{SSU}$ manifold cannot be the target of any nonconstant stable $\Phi$-harmonic maps from any manifold, $\rm (b)$ The homotopic class of any map from any manifold into a compact $\Phi$-$\text{SSU}$ manifold contains elements of arbitrarily small $\Phi$-energy,
$(\rm c)$  A compact $\Phi$-$\text{SSU}$ manifold cannot be the domain of any nonconstant stable $\Phi$-harmonic map into any manifold, and $(\rm d)$ The homotopic class of any map from a compact $\Phi$-$\text{SSU}$ manifold into any manifold contains elements of arbitrarily small $\Phi$-energy (cf. Theorem $1.1 (a),(b),(c)$, and $(d)$.) We also provide many examples of $\Phi$-SSU manifolds, and establish a link of  $\Phi$-SSU manifold to $p$-SSU manifold and topology.
The extrinsic average variational method in the calculus of variations that we have employed is in contrast to an average method in PDE that we applied in \cite {CW} to obtain sharp growth estimates for warping functions in multiply warped product
manifolds.
\end{abstract}

\maketitle

\section{Introduction}
Symmetric $2$-covariant tensor fields $\alpha$ on a Riemannian manifold $M$ of dimension $m$ such as the Riemannian metric of $M$, the Ricci tensor of $M$, a second fundamental form (for a given direction) of an immersion of $M\, ,$ the pull back metric tensor $u^{\ast} h$ on $M$ from  a smooth map $u: (M, g) \to (N,h)\, ,$ $F$-stress energy tensor of $u$ (where $F: [0, \infty) \to [0, \infty)$ is a $C^2$ strictly increasing function with $F(0)=0,$ cf. \cite {DW}), etc are of fundamental importance. At any fixed point $x_0 \in M\, ,$ $\alpha$ has the eigenvalues $\lambda$ relative to the metric $g$ of $M$; i.e., the $m$ real roots of the equation $\det (g_{ij} \lambda - \alpha_{ij}) = 0\, $ where $g_{ij} = g(e_i,e_j)$,  $\alpha_{ij} = \alpha(e_i,e_j)\, ,$ and $\{e_1, \cdots e_m\}$ is a basis for $T_{x_0}(M)\, .$  The {\it algebraic} invariants - the $k$-th elementary symmetric function of the eigenvalues of $\alpha$ at $x_0$, denoted by $\sigma_k (\alpha_{x_0}), 1 \le k \le m $ frequently have {\it geometric} meaning of the manifold $M$ or the map $u$ on $M$ with analytic, topological and physical impacts. For example, if we take $\alpha$ to be the Ricci tensor of $M\, ,$ then $\sigma _1 (\alpha)$  is
the scalar curvature of $M$ and is a central theme of Yamabi problem (\cite {T,A,S,J}) and conformal geometry (e.g. \cite{CY}, \cite {DLW}). If we take $\alpha$ to be the above second fundamental form, then  $\sigma _1(\alpha)$ and $\sigma _m(\alpha)$ are the mean curvature and the Gauss-Kronecker curvature (for that given direction) respectively. In the study of prescribed curvature problems in PDE, the existence of closed starshaped hypersurfaces of prescribed mean curvature in Euclidean space was proved by A.E. Treibergs and S.W. Wei \cite {TW}, solving a problem of F. Almgren and S.T. Yau \cite {Y}. While the case of prescribed Guass-Kronecker curvature was studied by V.I. Oliker \cite {O} and P. Delano\"e \cite{D}, the case of prescribed $k$-th mean curvature, in particular the intermediate cases, $2 \le k \le m-1$ were treated by L. Caffarelli, L. Nirenberg and J. Spruck \cite{CNS}.

On the other hand, from the viewpont of geometric mapping theory, the energy density $e(u)$ of $u$, the $p$-energy density $e_p(u)$ of $u$ and the $F$-energy density $e_F(u)$ of $u$  are  $\frac 12 \sigma _1(\alpha)$ , $\frac 1p \big (\sigma _1(\alpha )^{\frac p2}\big )\, $ and $F \circ \big ( \sigma _1(\alpha)\big )\, $ respectively, where  $\alpha$ is $u^{\ast} h\,
$, $1 \le p < \infty$, and  $F$ is the function as above.

In this paper, we define {\it $\Phi$-energy density} $e_{\Phi}(u)$ of $u$ to be a quarter of the second symmetric function $\sigma _2$ of $\alpha\, ,$ given by \begin{equation}\label{1.1} e_{\Phi}(u) = \frac 14 \sigma _2 (\alpha)\, , \, \text{where} \,  \alpha = u^{\ast} h\,
; \, \text{i.e.}\, ,  e_{\Phi}(u) = \frac 14 \sum_{i,j=1}^m\langle du(e_i),du(e_j)\rangle ^2\, .\end{equation} Here $\{e_1, \cdots e_m\}$ is a local orthonormal frame field on $M\, ,$ and $du$ is the differential of $u\, .$ Just as the energy $E(u)\, ,$ the $p$-energy $E_p(u)\, $ and the $F$-energy $E_{F}(u)$ of $u$ are the integrals of the energy density $e(u)$ of $u\, ,$ the $p$-energy density $e_p(u)$ of $u\, $ and the $F$-energy density $e_{F}(u)$ of $u\, $ respectively over the source manifold $M$ with the volume element $dx$, so we define
the $\Phi$-energy $E_{\Phi} (u)$ of $u$ to be \begin{equation}E_{\Phi} (u) = \int _M e_{\Phi}(u) \, dx\,  .\end{equation} Similarly, just as  $u$ is said to be {\it harmonic, $p$-harmonic, and $F$-harmonic} if it is a critical point of the energy functional $E(u)\, ,$  the $p$-energy functional $E_p(u)\, $ and the $F$-energy functional  $E_F(u)$ of $u$ respectively with respect to any smooth, compactly supported variation of $u$, so we make the following.

\begin{defn}
A smooth map $u$ is said to be {\it $\Phi$-harmonic } if it is a critical point of the $\Phi$-energy functional $E_{\Phi}$ with respect to any smooth compactly supported variation of $u\, ,$ stable $\Phi$-harmonic or simply $\Phi$-stable if $u$ is a local minimum of $E_{\Phi} (u)\, ,$ and $\Phi$-unstable if $u$ is not $\Phi$-stable.
\end{defn}
\bigskip

 We apply \emph {an extrinsic average variational method} in
the calculus of variations (\cite {W0})
and find a large class of manifolds of positive Ricci curvature that enjoy rich properties, and introduce the notions of \emph {superstrongly unstable $(\operatorname{SSU})$ manifolds} and \emph {$p$-superstrongly unstable $(p$-$\operatorname{SSU})$ manifolds} (\cite {WY,13,15}).
\begin{defn}
A Riemannian manifold $N$ with its Riemannian metric $\langle \, , \, \rangle _N$ is said to be
{\bf superstrongly unstable (SSU)} , if there exists an
isometric immersion of $N$ in $(\mathbb R^q, \langle \, \cdot\, \rangle _{\mathbb R^q})$ with its second fundamental form $\mathsf B$, such that for every {\sl unit} tangent vector $\mathsf x$
to $N$ at every point $y\in N$, the following symmetric linear operator $Q^N_y$
is {\sl negative definite}.
\begin{equation}\label{1.3}\langle Q^N_y(\mathsf x),\mathsf x\rangle_N=\sum^n_{\beta=1} \bigg (2
\langle \mathsf B(\mathsf x,\mathsf e_{\beta}), \mathsf B(\mathsf x,\mathsf e_{\beta})\rangle _{\mathbb R^q}  -
\langle \mathsf B(\mathsf x,\mathsf x), \mathsf B(\mathsf e_{\beta},\mathsf e_{\beta})\rangle _{\mathbb R^q} \bigg )\end{equation}
and $N$ is said to be
{\bf $\bold p$-superstrongly unstable ($p$-SSU)} for $p\geq 2$ if the following
functional is {\sl negative valued}.
\begin{equation}\label{1.4}\mathsf F_{p,y}(\mathsf x)=(p-2)\langle \mathsf B(\mathsf x,\mathsf x),\mathsf B(\mathsf x,\mathsf x)\rangle _{\mathbb R^q} + \langle Q^N_y(\mathsf x),\mathsf x\rangle _N,\end{equation}
where $\lbrace\mathsf e_1,\ldots,\mathsf e_n\rbrace$ is an orthonormal frame on $N$.
\end{defn}

In this paper we show that the extrinsic average variational method in
the calculus of variations employed in the study of harmonic maps, $p$-harmonic maps, $F$-harmonic maps and Yang-Mills fields can be extended to the study of $\Phi$-harmonic maps. In fact, we find a large class of manifolds with rich properties, \emph {$\Phi$-superstrongly unstable $(\Phi$-$\operatorname{SSU})$ manifolds}, establish their links to $p$-$\operatorname{SSU}$ manifolds and topology, and apply the theory of $p$-harmonic maps, minimal varieties and Yang-Mills fields to study such manifolds. With the same notations as above, we introduce the following notions:

\begin{defn}\label{D:1.3}
A Riemannian $n$-manifold $N$ is said to be $\Phi$-supersrongly unstable $(\Phi$-$\text{SSU})$ if there exists an isometric immersion of $N$ in $\mathbb R^q$ with its second fundamental form $\mathsf B$ such that, for all unit tangent vectors $\mathsf x$ to $N$ at every point $y\in N$, the following functional is always negative-valued:
\begin{equation}\label {1.5}
\mathsf F_{y}(\mathsf x)=\sum_{\beta=1}^n\bigg (4\langle \mathsf B(\mathsf x,\mathsf e_{\beta}),\mathsf B(\mathsf x,\mathsf e_{\beta})\rangle _{\mathbb R^q} - \langle \mathsf B(\mathsf x,\mathsf x),\mathsf B(\mathsf e_{\beta},\mathsf e_{\beta})\rangle _{\mathbb R^q}\bigg ),
\end{equation}
\end{defn}

Examples of  $\Phi$-$\text{SSU}$ manifolds include $n$-dimensional elliptic paraboloid in $\mathbb R^{n+1}$, $\{(x_1, \dots, x_n, y): y=x_1^2+\dots +x_n^2\}\, ,$ the standard $n$-sphere $S^n\, ,$ for $n>4$, certain minimal submanifolds in ellipsoids and in convex hypersurfaces, etc. (cf. Section 7.)

Furthermore, we prove, in particular,

\begin{thm} If $N$ is a compact $\Phi$-$\operatorname{SSU}$ manifold, then
\begin{enumerate}
\item[{\rm (a)}] For every compact manifold $M$, there are no nonconstant smooth stable $\Phi$-harmonic map $u: M \to N\, .$
\item[{\rm (b)}] The homotopic class of any map from $M$ into $N$ contains elements of arbitrarily small $\Phi$-energy.
\item[{\rm (c)}]For every compact manifold $\tilde N$, there are no nonconstant smooth stable $\Phi$-harmonic map $u: N \to \tilde N\, .$
\item[{\rm (d)}] The homotopic class of any map from $N$ into $\tilde N$ contains elements of arbitrarily small $\Phi$-energy.
\end{enumerate}\label{T:1.1}
\end{thm}

The cases $(1)\, N$ is $S^n, n \ge 5$ and $(2)\, N$ is a minimal submanifold in the unit sphere with $\text{Ric}^N\geq \frac{3}{4}n$ satisfying properties $(a)$ and $(c)$ are due to S. Kawai and N. Nakauchi (cf. \cite{KN, KN2}).
These are analogs of the following: $S^n, n > 2$ is not the domain of any nonconstant stable harmonic maps into any Riemannian manifold due to Xin (\cite {9}), $S^n, n > 2$ is not the target of any nonconstant stable harmonic maps from any Riemannian manifold due to
Leung \cite {12} and Wei \cite {15}, and a minimal $k$-submanifold $N$ in the unit sphere with $\operatorname{Ric}^N > (1- \frac{1}{p}) k , p < k $ is neither the domain nor the target of any nonconstant stable $p$-harmonic maps (cf. \cite{WY}).

For brevity we call such a manifold with properties $(a)$, $(b)$, $(c)$ and $(d)$, $\Phi$-strongly unstable $(\Phi$-$\text{SU})$. That is,
\begin{defn}
A Riemannian manifold $N$ is $\Phi$-strongly unstable $(\Phi$-$\text{SU})$ if it is neither the domain nor the target of any nonconstant smooth $\Phi$-stable harmonic map, and the homotopic class of maps from or into $N$ contains a map of arbitrarily small energy.
\end{defn}
This leads to the study of the identity map on a Riemannian manifold. In particular, if $N$ is $\Phi$-$\operatorname{SU}$, then the identity map of $N$ is $\Phi$-unstable. For convenience, we make the following
\begin{defn}
A Riemannian manifold $N$ is $\Phi$-unstable $(\Phi$-$\text{U})$ if the identity map $\text{Id}_N$ on $N$ is $\Phi$-unstable.
\end{defn}
\noindent
and obtain the following results.
\begin{thm}(cf. Section 4) Let $N$ be a compact manifold. Then
\noindent
$($A$)$ $N$ is $\Phi$-$\operatorname{SSU}$ $\Rightarrow$ $($B$)$ $N$ is $\Phi$-$\operatorname{SU}$ $\Rightarrow$ $($C$)$ $N$ is $\Phi$-$\operatorname{U}\, .$
\end{thm}\smallskip

What seems to be remarkable is that the above results $($A$)$ $\Rightarrow$ $($B$)$ $\Rightarrow$ $($C$)$ go the other way around on certain compact homogeneous spaces and are in sharp contrast to $p$-harmonic maps where there are gap phenomena that dash the hope for  $($C$)$ $\Rightarrow$ $($B$)$ $\Rightarrow$ $($A$)\, $ (cf. \cite {13}).\smallskip

\noindent
{\bf Theorem \ref{T:9.1}}
{\it Let $N=G/H$ be a compact irreducible homogeneous space of dimension $n$ with first eigenvalues $\lambda_1$ and scalar curvature $\text{Scal} ^{\, N}$. Set the following properties $($A$)$ through $($D$)$:\smallskip

\noindent
$($A$)$ $N$ is $\Phi$-$\operatorname{SSU}$.\\
$($B$)$ $N$ is $\Phi$-$\operatorname{SU}$.\\
$($C$)$ $N$ is $\Phi$-$\operatorname{U}$.\\
$($D$)$ $\lambda_1<\frac{4}{3n}\text{Scal} ^{\, N}$.\smallskip

\noindent
Then the following holds:\smallskip

\noindent$($A$)$ $\Leftrightarrow$ $($B$)$ $\Leftrightarrow$ $($C$)$ $\Leftrightarrow$ $($D$)\, .$}
\smallskip

Furthermore, we establish a link of  $\Phi$-$\operatorname{SSU}$ manifold to $p$-$\operatorname{SSU}$ manifold and topology:
\smallskip

\noindent
{\bf Theorem \ref{T: 6.1}}
{\it Every $\Phi$-$\operatorname{SSU}$ manifold is $p$-$\operatorname{SSU}$  for any $2 \le p \le 4$ and every compact $\Phi$-$\operatorname{SSU}$ manifold is $4$-connected, i.e. $\pi_1(N) = \cdots = \pi_{4}(N) = 0$.}
\smallskip

\noindent
{\bf Theorem 7.2 $($Sphere Theorem$)$}
{\it Every compact $\Phi$-$\operatorname{SSU}$ manifold with dimension $n < 10$ is homeomorphic to the $n$-sphere $S^n$.}
\smallskip

The extrinsic average variational method in the calculus of variations is in contrast to an average method in PDE that we applied in \cite {CW} to obtain sharp growth estimates for warping functions in multiply warped product
manifolds.

\section{Fundamentals of $\Phi$-harmonic Maps}

Let $u: M \to N$ be a smooth map between compact Riemannian manifolds of dimension $m$ and $n$ respectively, $T^{\ast}M$ be the cotangent bundle of $M\, ,$ and $TN$ be the tangent bundle of $N\, .$ We denote $u^{-1}TN = \{(x,v) \in M \times TN : u(x) = \pi (v) \}\, ,$ the pull-back bundle, that is the vector bundle over $M$ induced by $u$ from the tangent bundle $\pi : TN \to N\, .$ Then the differential $du$  of $u\, $ is a differentiable $1$-form with values in the pull-back bundle $u^{-1}TN\, ,$  or $du \in \Gamma (TM^{\ast} \bigotimes u^{-1}TN)\, $ is a section of the bundle $TM^{\ast} \bigotimes u^{-1}TN \to M$. To simplify the notation, let  $\langle \mathsf Y, \mathsf Z \rangle = \langle \mathsf Y, \mathsf Z \rangle _N$ for all vector fields $\mathsf Y$ and $\mathsf Z$ on $N\, ,$ and $\nabla^u$ be the pull-back connection. 
\smallskip
Choose a compactly supported  one-parameter $C^2$ family of $C^1$ maps $\Psi(\cdot, t) =u_t (\cdot), -\varepsilon< t <\varepsilon$ such that $\Psi(\cdot, 0) =u_0 (\cdot) = u(\cdot)$ and
${{du_t}\over{dt}}_{\big |_{t=0}}=v$ is $C^1$ and a two-parameter $C^1$ variations
$\Psi(\cdot,s,t)=u_{s,t}, -\varepsilon< s,t <\varepsilon$ such that
$$V={{\partial u_{s,t}}\over{\partial s}},\quad
v={{\partial u_{s,t}}\over{\partial s}}_{\big |_{(s,t)=(0,0)}},\quad
W={{\partial u_{s,t}}\over{\partial t}}\quad \text {and}\quad w=
{{\partial u_{s,t}}\over{\partial t}}_{\big |_{(s,t)=(0,0)}}.$$

\noindent We shall denote the pull-back connection
by $\nabla^{\psi}$ from $\Psi$. 

\smallskip
\begin{prop} [First variation formula for $\Phi$-energy $E_{\Phi}$]
\[\frac {d}{dt} E_{\Phi}(u_t)=- \int _M\bigg \langle V,\sum_{i,j=1}^m \nabla ^{\Psi}_{e_i} \big (\langle d\Psi (e_i),d\Psi (e_j) \rangle d\Psi (e_j) \big ) \bigg \rangle\, dx,
\]
where $\{e_1,\cdots,e_m\}$ is a local orthonormal frame field on $M$, and
$V=d\Psi(\frac {\partial}{\partial t})$.\label{P:2.1}
\end{prop}

\begin{proof}
Since the Lie bracket $[\frac {\partial}{\partial t},e_i]  = 0\, ,$
\begin{equation}\label{2.1}
\begin{aligned}
\nabla ^{\Psi}_{\frac {\partial}{\partial t}}\big ( d\Psi (e_i) \big )
& =  \nabla ^{\Psi}_{e_i}\,  \big (d\Psi (\frac {\partial}{\partial t})\big ) + d\Psi ( [\frac {\partial}{\partial t},e_i] )\\
& = \nabla ^{\Psi}_{e_i}\,  V\, .
\end{aligned}
\end{equation}
Since for every $f \in C^{\infty} (M)\, , $
 \begin{equation}\label{2.50}e_i (f) = \langle \text{grad}_M f, e_i \rangle _M = \text{div}_M (f e_i)
 \end{equation} at a point in $M\, ,$ \eqref{1.1}, \eqref{2.1} and \eqref{2.50} imply that
 \begin{equation}\label{2.2}
\begin{aligned}
\frac {\partial} {\partial t}  e_{\Phi}(u_t) & = \frac 14 \frac {\partial} {\partial t}  \sum_{i,j=1}^m\langle d\Psi (e_i),d\Psi (e_j)\rangle ^2\\
& = \sum_{i,j=1}^m \langle d\Psi (e_i),d\Psi (e_j) \rangle \langle \nabla ^{\Psi}_{\frac {\partial}{\partial t}}\big ( d\Psi (e_i) \big ), d\Psi (e_j) \rangle \\
& = \sum_{i,j=1}^m \langle \nabla ^{\Psi}_{e_i} V, d\Psi (e_j) \rangle \langle d\Psi (e_i),d\Psi (e_j) \rangle \\
&  =\sum_{i=1}^m \bigg \langle \nabla ^{\Psi}_{e_i}\,V,\sum_{j=1}^m \langle d\Psi (e_i),d\Psi (e_j) \rangle d\Psi (e_j) \bigg \rangle \\
& =\sum_{i=1}^me_i \bigg \langle V,\sum_{j=1}^m \langle d\Psi (e_i),d\Psi (e_j) \rangle d\Psi (e_j) \bigg \rangle\\
&\quad \quad -\sum_{i=1}^m \bigg \langle \,V,\sum_{j=1}^m \nabla ^{\Psi}_{e_i} \big (\langle d\Psi (e_i),d\Psi (e_j) \rangle d\Psi (e_j) \big ) \bigg \rangle  \\
& =\text{div}_M \bigg ( \sum_{i=1}^m\bigg \langle V,\sum_{j=1}^m \langle d\Psi (e_i),d\Psi (e_j) \rangle d\Psi (e_j) \bigg \rangle e_i \bigg )\\
& \quad \quad- \bigg \langle \,V,\sum_{i,j=1}^m \nabla ^{\Psi}_{e_i} \big (\langle d\Psi (e_i),d\Psi (e_j) \rangle d\Psi (e_j) \big ) \bigg \rangle.
\end{aligned}
\end{equation}
Since $\frac {d}{dt} E_{\Phi}(u_t) = \int_M \frac {\partial} {\partial t}e_{\Phi}(u_t)\, dx\, ,$ integrating both sides of  \eqref{2.2} and applying the divergence theorem, we
obtain the desired.
\end{proof}
\begin{cor}$($\cite {KN}$)$
\begin{eqnarray}\label{2.3}
\qquad \frac{dE_{\Phi}(u_t)}{dt}_{\big |_{t=0}} = - \int _M   \bigg \langle \,  v,\sum_{i,j=1}^m \nabla ^{u}_{e_i} \big (\langle  du (e_i),du (e_j) \rangle du (e_j) \big ) \bigg \rangle\, dx.
\end{eqnarray}
\label{C:2.1}
\end{cor}

\begin{cor}\label{C:2.2}
A smooth map $u$ is $\Phi$-harmonic if and only if $u$ satisfies \begin{equation}\label{2.4}
\sum_{i,j=1}^m \nabla ^{u}_{e_i} \big (\langle  du (e_i),du (e_j) \rangle du (e_j) \big ) = 0.\end{equation}
\end{cor}

\begin{proof} This follows at once from the definition of $\Phi$-harmonic map and Proposition \ref{P:2.1}, the first variation formula of $\Phi$-energy.
\end{proof}

\begin{cor}[Example of $\Phi$-harmonic]
The identity map on any Riemannian manifold is $\Phi$-harmonic.  \label{C:2.3}
\end{cor}

\begin{proof} If $u$ is the identity map on $M$, then $N=M, m=n$, $du(e_i) = e_i\, ,$ and 
$$\sum_{i,j=1}^m \nabla ^{u}_{e_i} \big (\langle  du (e_i),du (e_j) \rangle du (e_j) \big ) = \sum_{i=1}^m  \nabla ^{Id}_{e_i}  ( \sum_{j=1}^m  \delta _{ij} e_j  ) = 0\, .$$
Consequently, $u$ is $\Phi$-harmonic by Corollary \ref{C:2.2}.
\end{proof}

\begin{prop}  $\big ($The second variation formula of two parameters for  $\Phi$-energy $E_{\Phi}$($u$ is not necessary $\Phi$-harmonic)$\big )$
\begin{equation}\label{2.5}
\begin{aligned}
& \frac {\partial ^2} {\partial s \partial t}  E_{\Phi}(u_{s,t}) \bigg ( = \frac 14 \frac {\partial ^2} {\partial s \partial t}  \int _M \sum_{i,j=1}^m\langle d\Psi (e_i),d\Psi (e_j)\rangle ^2\, dx \bigg)\\
& =  \int _M \sum_{i,j=1}^m \langle \nabla ^{\Psi}_{e_i} V , d\Psi (e_j) \rangle \langle \nabla ^{\Psi}_{e_i} W, d\Psi (e_j) \rangle\, dx\\
 &\quad+ \int _M \sum_{i,j=1}^m \langle  d\Psi (e_i), \nabla ^{\Psi}_{e_j} W  \rangle \langle \nabla ^{\Psi}_{e_i} V, d\Psi (e_j) \rangle\, dx\\
& \quad + \int_M  \sum_{i,j=1}^m \langle  d\Psi (e_i) , d\Psi (e_j) \rangle \langle \nabla ^{\Psi}_{e_i} V, \nabla ^{\Psi}_{e_j} W\rangle \, dx \\
& \quad -  \int _M \bigg \langle \,  \nabla ^{\Psi}_{\frac {\partial}{\partial s}}  V,\sum_{i,j=1}^m \nabla ^{\Psi}_{e_i} \big (\langle  d\Psi (e_i),d\Psi (e_j) \rangle d\Psi (e_j) \big ) \bigg \rangle\, dx \\
& \quad  + \int _M \sum_{i=1}^m  \bigg \langle R^{N} \big (V, d\Psi (e_i) \big ) W, \sum_{j=1}^m \langle  d\Psi (e_i) , d\Psi (e_j) \rangle d\Psi (e_j) \bigg \rangle \, dx,
\end{aligned}
\end{equation}
where $R^N$ is the curvature tensor of $N$, and $\langle R(x,y)y, x \rangle $ denotes the sectional curvature of the plane spanned by $\{x,y\}\, .$\label{P:2.2}
\end{prop}
\begin{proof}
By \eqref{2.1},
\begin{equation}\label{2.6}
\begin{aligned}
&\langle \nabla ^{\Psi}_{\frac {\partial}{\partial s}} \nabla ^{\Psi}_{\frac {\partial}{\partial t}}\big ( d\Psi (e_i) \big ),  d\Psi (e_j) \rangle\\
 & = \langle \nabla ^{\Psi}_{\frac {\partial}{\partial s}} \nabla ^{\Psi}_{e_i} V,  d\Psi (e_j) \rangle   \\
&  = \langle \nabla ^{\Psi}_{e_i} \nabla ^{\Psi}_{\frac {\partial}{\partial s}} V,  d\Psi (e_j) \rangle +
 \langle R^{N} \big (d\Psi ({\frac {\partial}{\partial s}} ), d\Psi (e_i)\big ) V,  d\Psi (e_j) \rangle \\
&  = \langle \nabla ^{\Psi}_{e_i} \nabla ^{\Psi}_{\frac {\partial}{\partial s}} V,  d\Psi (e_j) \rangle +
 \langle R^{N} \big (V, d\Psi (e_j)\big ) W,  d\Psi (e_i) \rangle.
\end{aligned}
\end{equation}
This via \eqref{2.50} implies
\begin{equation}\label{2.7}
\begin{aligned}
&  \sum_{i,j=1}^m \langle  d\Psi (e_i) , d\Psi (e_j) \rangle \langle \nabla ^{\Psi}_{\frac {\partial}{\partial s}} \nabla ^{\Psi}_{\frac {\partial}{\partial t}}\big ( d\Psi (e_i) \big ),  d\Psi (e_j) \rangle \\
& =   \sum_{i,j=1}^m \langle  d\Psi (e_i) , d\Psi (e_j) \rangle \big ( \langle \nabla ^{\Psi}_{e_i} \nabla ^{\Psi}_{\frac {\partial}{\partial s}} V,  d\Psi (e_j) \rangle +
 \langle R^{N} \big (V, d\Psi (e_j)\big ) W,  d\Psi (e_i) \rangle \big )\\
 & =  \sum_{i,=1}^m   \bigg \langle \nabla ^{\Psi}_{e_i} \nabla ^{\Psi}_{\frac {\partial}{\partial s}} V ,  \sum_{j=1}^m \langle  d\Psi (e_i) , d\Psi (e_j) \rangle d\Psi (e_j) \bigg \rangle\\
 &\quad+ \sum_{i=1}^m  \bigg< R^{N} \big (V, d\Psi (e_i)\big ) W, \sum_{j=1}^m \langle  d\Psi (e_i) , d\Psi (e_j) \rangle d\Psi (e_j) \bigg> \\
& = \sum_{i=1}^m  e_i \bigg<   \nabla ^{\Psi}_{\frac {\partial}{\partial s}}  V,\sum_{j=1}^m \langle  d\Psi (e_i),d\Psi (e_j) \rangle d\Psi (e_j) \bigg> \\
&\quad-  \sum_{i=1}^m  \bigg< \,  \nabla ^{\Psi}_{\frac {\partial}{\partial s}}  V,\sum_{j=1}^m \nabla ^{\Psi}_{e_i} \big (\langle  d\Psi (e_i),d\Psi (e_j) \rangle d\Psi (e_j) \big ) \bigg> \\
& \quad +
\sum_{i=1}^m  \bigg< R^{N} \big (V, d\Psi (e_i)\big ) W, \sum_{j=1}^m \langle  d\Psi (e_i) , d\Psi (e_j) \rangle d\Psi (e_j) \bigg> \\
& =  \text{div}_M \bigg ( \sum_{i=1}^m  \bigg {\langle}  \nabla ^{\Psi}_{\frac {\partial}{\partial s}} V,\sum_{j=1}^m \langle  d\Psi (e_i),d\Psi (e_j) \rangle d\Psi (e_j) \bigg {\rangle} e_i \bigg ) \\
& \quad -   \bigg<  \,  \nabla ^{\Psi}_{\frac {\partial}{\partial s}} V,\sum_{i,j=1}^m \nabla ^{\Psi}_{e_i} \big (\langle  d\Psi (e_i),d\Psi (e_j) \rangle d\Psi (e_j) \big ) \bigg> \\
& \quad +
\sum_{i=1}^m  \bigg< R^{N} \big (V, d\Psi (e_i)\big ) W, \sum_{j=1}^m \langle  d\Psi (e_i) , d\Psi (e_j) \rangle d\Psi (e_j) \bigg>.
\end{aligned}
\end{equation}

In view of \eqref{2.1} and \eqref{2.7} we have

\begin{equation}\label{2.8}
\begin{aligned}
& \frac {\partial ^2} {\partial s \partial t}  e_{\Phi}(u_{s,t}) \\
& = \frac 14 \frac {\partial ^2} {\partial s \partial t}  \sum_{i,j=1}^m\langle d\Psi (e_i),d\Psi (e_j)\rangle ^2\\
& = \frac {\partial } {\partial s} \sum_{i,j=1}^m \langle d\Psi (e_i),d\Psi (e_j) \rangle \langle \nabla ^{\Psi}_{\frac {\partial}{\partial t}}\big ( d\Psi (e_i) \big ), d\Psi (e_j) \rangle  \\
& =  \sum_{i,j=1}^m \langle \nabla ^{\Psi}_{\frac {\partial}{\partial s}}\big ( d\Psi (e_i) \big ), d\Psi (e_j) \rangle \langle \nabla ^{\Psi}_{\frac {\partial}{\partial t}}\big ( d\Psi (e_i) \big ), d\Psi (e_j) \rangle \\
&\quad  +  \sum_{i,j=1}^m \langle d\Psi (e_i), \nabla ^{\Psi}_{\frac {\partial}{\partial s}}\big ( d\Psi (e_j) \big ) \rangle \langle \nabla ^{\Psi}_{\frac {\partial}{\partial t}}\big ( d\Psi (e_i) \big ), d\Psi (e_j) \rangle \\
& \quad +   \sum_{i,j=1}^m \langle  d\Psi (e_i) , d\Psi (e_j) \rangle \langle \nabla ^{\Psi}_{\frac {\partial}{\partial t}}\big ( d\Psi (e_i) \big ), \nabla ^{\Psi}_{\frac {\partial}{\partial s}} \big ( d\Psi (e_j) \big )\rangle\\
& \quad  +   \sum_{i,j=1}^m \langle  d\Psi (e_i) , d\Psi (e_j) \rangle \langle \nabla ^{\Psi}_{\frac {\partial}{\partial s}} \nabla ^{\Psi}_{\frac {\partial}{\partial t}}\big ( d\Psi (e_i) \big ),  d\Psi (e_j) \rangle \\
& =  \sum_{i,j=1}^m \langle \nabla ^{\Psi}_{e_i} V , d\Psi (e_j) \rangle \langle \nabla ^{\Psi}_{e_i} W, d\Psi (e_j) \rangle + \sum_{i,j=1}^m \langle  d\Psi (e_i), \nabla ^{\Psi}_{e_j} W  \rangle \langle \nabla ^{\Psi}_{e_i} V, d\Psi (e_j) \rangle\\
& \quad+   \sum_{i,j=1}^m \langle  d\Psi (e_i) , d\Psi (e_j) \rangle \langle \nabla ^{\Psi}_{e_i} V, \nabla ^{\Psi}_{e_j} W\rangle\\
  &\quad+  \text{div}_M \bigg ( \sum_{i=1}^m  \bigg \langle  \nabla ^{\Psi}_{\frac {\partial}{\partial s}} V,\sum_{j=1}^m \langle  d\Psi (e_i),d\Psi (e_j) \rangle d\Psi (e_j) \bigg \rangle e_i \bigg ) \\
\end{aligned}
\end{equation}
 \begin{equation*}
\begin{aligned}
& \quad -  \bigg< \,  \nabla ^{\Psi}_{\frac {\partial}{\partial s}}  V,\sum_{i,j=1}^m \nabla ^{\Psi}_{e_i} \big (\langle  d\Psi (e_i),d\Psi (e_j) \rangle d\Psi (e_j) \big ) \bigg> \\
& \quad +
\sum_{i=1}^m  \left < R^{N} \big (V, d\Psi (e_i) \big ) W, \sum_{j=1}^m \langle  d\Psi (e_i) , d\Psi (e_j) \rangle d\Psi (e_j) \right > .\\
\end{aligned}
\end{equation*}

Since $ \frac {\partial ^2} {\partial s \partial t}  E_{\Phi}(u_{s,t}) = \int _M \frac {\partial ^2} {\partial s \partial t}  e_{\Phi}(u_{s,t}) \, dx\, ,$ integrating both sides of \eqref{2.8} and using the divergence theorem, we obtain the desired.
\end{proof}

As an immediate consequence, we obtain the following.
\begin{cor}  $\big ($Two parameter variation formula of $\Phi$-energy $E_{\Phi}$($u$ is not necessary $\Phi$-harmonic)$\big )$
\begin{equation}\label{2.9}
\begin{aligned}
& \frac {\partial ^2} {\partial s \partial t}  E_{\Phi}(u_{s,t})_{\big |_{(s,t)=(0,0)}}  \bigg ( = \frac 14 \frac {\partial ^2} {\partial s \partial t}  \int _M \sum_{i,j=1}^m\langle d\Psi (e_i),d\Psi (e_j)\rangle ^2\, dx_{\big |_{(s,t)=(0,0)}}  \bigg)\\
& =  \int _M \sum_{i,j=1}^m \langle \nabla ^{u}_{e_i} v , du (e_j) \rangle \langle \nabla ^{u}_{e_i} w, du (e_j) \rangle\, dx\\
 &\quad+ \int _M \sum_{i,j=1}^m \langle  du (e_i), \nabla ^{u}_{e_j} w  \rangle \langle \nabla ^{u}_{e_i} v, du (e_j) \rangle\, dx\\
& \quad + \int_M  \sum_{i,j=1}^m \langle  du (e_i) , du (e_j) \rangle \langle \nabla ^{u}_{e_i} v, \nabla ^{u}_{e_j} w\rangle \, dx\\
 &\quad-  \int _M \bigg< \,  \nabla ^{u}_{\frac {\partial}{\partial s}}  v,\sum_{i,j=1}^m \nabla ^{u}_{e_i} \big (\langle  du (e_i),du (e_j) \rangle du (e_j) \big ) \bigg>\, dx \\
& \quad + \int _M \sum_{i=1}^m  \bigg< R^{N} \big (v, du (e_i) \big ) w, \sum_{j=1}^m \langle  du (e_i) , du (e_j) \rangle du (e_j) \bigg>\, dx.\\
\end{aligned}
\end{equation}\label{C:2.4}
\end{cor}
\begin{proof} This follows at once from Proposition \ref{P:2.2}.
\end{proof}
\begin{cor}  $\big ($The second variation formula of $\Phi$-energy $E_{\Phi}$($u$ is not necessary $\Phi$-harmonic)$\big )$
\begin{equation}\label{2.10}
\begin{aligned}
& \frac {d^2} {dt^2}  E_{\Phi}(u_{t})_{\big |_{t=0}}
 =  \int _M \sum_{i,j=1}^m \langle \nabla ^{u}_{e_i} v , du (e_j) \rangle ^2\, dx
+ \int _M \sum_{i,j=1}^m \langle  du (e_i), \nabla ^{u}_{e_j} v  \rangle \langle \nabla ^{u}_{e_i} v, du (e_j) \rangle\, dx
\end{aligned}
\end{equation}
\begin{equation*}
\begin{aligned}
& \quad + \int_M  \sum_{i,j=1}^m \langle  du (e_i) , du (e_j) \rangle \langle \nabla ^{u}_{e_i} v, \nabla ^{u}_{e_j} v \rangle \, dx\\
&\quad -  \int _M \bigg< \,  \nabla ^{u}_{\frac{\partial}{\partial t}}  v,\sum_{i,j=1}^m \nabla ^{u}_{e_i} \big (\langle  du (e_i),du (e_j) \rangle du (e_j) \big ) \bigg>\, dx \\
& \quad + \int _M \sum_{i=1}^m  \bigg< R^{N} \big (v, du (e_i) \big ) v, \sum_{j=1}^m \langle  du (e_i) , du (e_j) \rangle du (e_j) \bigg> \, dx.
\end{aligned}
\end{equation*}\label{C:2.5}
\end{cor}
\begin{cor} Suppose either for each fixed $x_0 \in M$,
the curve $\Psi(x_0,t)$
is a constant speed geodesic in $N$ or $u$ is a $\Phi$-harmonic map with
compactly supported $V(x,0)$ in the interior of $M$.  Then
\begin{equation}\label{2.11}
\begin{aligned}
& \frac {d^2} {dt^2} E_{\Phi}(u_{t})=\int _M \sum_{i,j=1}^m \langle \nabla ^{\Psi}_{e_i} V , du_t (e_j) \rangle^2\, dx
+ \int _M \sum_{i,j=1}^m \langle du_t (e_i), \nabla ^{\Psi}_{e_j} V\rangle \langle \nabla ^{\Psi}_{e_i} V, du_t (e_j) \rangle\, dx\\
& \quad + \int_M\sum_{i,j=1}^m \langle du_t (e_i) , du_t (e_j) \rangle \langle \nabla ^{\Psi}_{e_i} V, \nabla ^{\Psi}_{e_j} V\rangle \, dx \\
& \quad + \int _M \sum_{i=1}^m \bigg \langle R^{N} \big (V, du_t (e_i) \big ) V, \sum_{j=1}^m \langle du_t (e_i) , du_t (e_j) \rangle du_t (e_j) \bigg \rangle \, dx.\\¡±
\end{aligned}
\end{equation}
In particular,
\begin{equation}\label{2.12}
\begin{aligned}
& \frac {d^2} {dt^2}  E_{\Phi}(u_{t})_{\big |_{t=0}}  =  \int _M \sum_{i,j=1}^m \langle \nabla ^{u}_{e_i} v , du (e_j) \rangle^2\, dx
+ \int _M \sum_{i,j=1}^m \langle  du (e_i), \nabla ^{u}_{e_j} v  \rangle \langle \nabla ^{u}_{e_i} v, du (e_j) \rangle\, dx\\
& \quad + \int_M  \sum_{i,j=1}^m \langle  du (e_i) , du (e_j) \rangle \langle \nabla ^{u}_{e_i} v, \nabla ^{u}_{e_j} v\rangle \, dx \\
& \quad + \int _M \sum_{i=1}^m \bigg<  R^{N} \big (v, du (e_i) \big ) v, \sum_{j=1}^m \langle  du (e_i) , du (e_j) \rangle du (e_j) \bigg> \, dx.\\
\end{aligned}
\end{equation}\label{C:2.6}
\end{cor}

\begin{remark}
The case $u$ is $\Phi$-harmonic, \eqref{2.12} is due to Kawai and Nakauchi \cite{KN}. 
\end{remark}

\begin{proof}
Set $W=V$ and $s=t$ in Proposition \ref{P:2.2}.  Then the term
\[ -  \int _M \bigg \langle \,  \nabla ^{\Psi}_{\frac {\partial}{\partial t}}  V,\sum_{i,j=1}^m \nabla ^{\Psi}_{e_i} \big (\langle  d\Psi (e_i),d\Psi (e_j) \rangle d\Psi (e_j) \big ) \bigg \rangle\, dx \]
vanishes  because by the assumption, either the curves
are constant speed geodesics in
which $\nabla^u_{{\partial\over{\partial
t}}}V\equiv 0$, or $u$ is a $\Phi$-harmonic map, by the
first variational formula the whole term is zero . This proves \eqref{2.11}. Setting $t=0$ and $\Psi(\cdot, 0) = u(\cdot)\, ,$ in \eqref{2.11}, we prove  \eqref{2.12}. 
\end{proof}

\begin{cor}(\cite {KN}) Let $u: M \rightarrow N$ be a $\Phi$-harmonic map. Then
\begin{equation}\label{2.13}
\begin{aligned}
& \frac {\partial ^2} {\partial s \partial t}  E_{\Phi}(u_{s,t})_{\big |_{(s,t)=(0,0)}}  =  \int _M \sum_{i,j=1}^m \langle \nabla ^{u} _{e_i} v , du (e_j) \rangle \langle \nabla ^{u}_{e_i} w, du (e_j) \rangle\, dx \\
& \quad + \int _M \sum_{i,j=1}^m  \langle \nabla ^{u}_{e_i} v, du (e_j) \rangle \langle  du (e_i), \nabla ^{u}_{e_j} w  \rangle\, dx\\
&\quad+ \int_M  \sum_{i,j=1}^m \langle  du (e_i) , du (e_j) \rangle \langle \nabla ^{u}_{e_i} v, \nabla ^{u}_{e_j} w\rangle \, dx \\
&\quad + \int _M \sum_{i=1}^m  \bigg \langle R^{N} \big (v, du (e_i) \big ) w, \sum_{j=1}^m \langle  du (e_i) , du (e_j) \rangle du (e_j) \bigg \rangle \, dx.\\
\end{aligned}
\end{equation}
\end{cor}

\begin{proof} This follows at once from Corollaries \ref{C:2.2} and \ref{C:2.4}.
\end{proof}

\section{An Average variational method Part I: Average second variation formulas for $\Phi$-energy }
We assume  $M$ (resp. $N$) is isometrically immersed in the Euclidean space $\mathbb R^{q}$. Let $\overline{\nabla}$ be the standard flat connection on $\mathbb R^{q}$, $\nabla$ (resp. $\nabla ^N$) the Riemannian connection on $M$ (resp. $N$) and $B$ (resp. $\mathsf B$) the second fundamental form of $M$ (resp. $N$) in $\mathbb R^{q}$. These are related by
\begin{equation}\label{3.1}
\overline{\nabla}_XY=\nabla_XY+B(X,Y)\qquad \big (\operatorname{resp}. \overline{\nabla}_{\mathsf X} \mathsf Y=\nabla^N_{\mathsf X} \mathsf Y+\mathsf B(\mathsf X,\mathsf Y)\big ) ,
\end{equation}
where $X,Y$ (resp. $\mathsf X,\mathsf Y$) are smooth vector fields on $M$ (resp. $N$). If $T^{\bot} M$ (resp. $T^{\bot} N$) is the normal bundle of $M$ (resp. $N$) in $\mathbb R^{q}$, $\eta$ (resp. $\zeta$) is a smooth section of $T^\bot M$ (resp. $T^\bot N$), then the Weingarten map $A^\eta X$ (resp. $\mathsf A^\zeta \mathsf X$ ) and the connection $\nabla^{\bot}_X\eta$ (resp. ${\nabla^N}^{\bot}_{\mathsf X}\zeta$ ) in the normal bundle are defined by
\begin{equation}\label{3.2}
\overline{\nabla}_X\eta=-A^\eta X+\nabla^\bot _X\eta \qquad \big (\operatorname{resp}. \overline{\nabla}^N_{\mathsf X}\zeta=-\mathsf A^\zeta \mathsf X+{\nabla^N}^{\bot} _{\mathsf X}\zeta \big ),
\end{equation}
where $-A^ \eta X$ (resp. $-\mathsf A^ \zeta \mathsf X$) is the component tangent to $M$ (resp. $N$) and $\nabla^\bot _{X}\eta$ (resp. ${\nabla^N}^{\bot} _{\mathsf X}\zeta $) is normal to $M$ (resp. $N$). The tensors $A$ and $B$ (resp. $\mathsf A$ and $\mathsf B$) are related by
\begin{equation} \label{3.3}\langle A^\eta X, Y \rangle = \langle B(X, Y), \eta  \rangle \qquad \big (\operatorname{resp}. \langle \mathsf A^\zeta \mathsf X, \mathsf Y \rangle = \langle \mathsf  B(\mathsf X, \mathsf Y), \zeta  \rangle \big )
\end{equation}
For each $x\in M$, let $e_{m+1},\cdots,e_{q}$ be an orthonormal basis for the normal space $T^\bot M_x$ to $M$ at $x$. Define the Ricci tensor $\text{Ric}^M : T_x(M) \to T_x(M)$ by
\begin{eqnarray}\label{3.4} \text{Ric}^M (v) = \sum _{i=1}^m R(v, e_i) e_i
\end{eqnarray}
Define selfadjoint linear map $Q^M_x:T_x M\rightarrow T_x M$ by
\begin{eqnarray} \label{3.5}
Q^M_x=\sum_{\alpha=m+1}^{q}\big (2A^{e_{\alpha}} A^{e_{\alpha}}-\text{trace}(A^{e_{\alpha}})A^{e_{\alpha}}\big ),
\end{eqnarray}
Then the Guass curvature equation implies
\begin{equation} \label{3.6}
\text{Ric}^M-\sum_{\alpha=m+1}^{q}\text{tr}(A^{e_{\alpha}})A^{e_{\alpha}}+\sum_{\alpha=m+1}^{q}A^{e_{\alpha}} A^{e_{\alpha}} = 0\, .
\end{equation}
Using this in the definition of $Q^M$ yields
\begin{equation} \begin{aligned}\label{3.7}
Q^M&=\sum_{\alpha=m+1}^{q}\left (2A^{e_{\alpha}} A^{e_{\alpha}} -\text{tr}(A^{e_{\alpha}})A^{e_{\alpha}}\right )\\
&=-2\text{Ric}^M+\sum_{\alpha=m+1}^{q}\text{tr}(A^{e_{\alpha}})A^{e_{\alpha}}=-\text{Ric}^M+\sum_{\alpha=m+1}^{q}A^{e_{\alpha}} A^{e_{\alpha}},
\end{aligned}\end{equation}

Similarly, for each $y\in N$, let $\mathsf e_{n+1},\cdots,\mathsf e_{q}$ be an orthonormal basis for the normal space $T^\bot N_y$ to $N$ at $y$. Define the Ricci tensor $\text{Ric}^N : T_y(N) \to T_y(N)$,  and selfadjoint linear map $Q^N_y:T_y N\rightarrow T_y N$ analogously and yields
\begin{equation*} \begin{aligned}(3.7^{\prime})\quad
Q^N&=\sum_{\alpha=n+1}^{q}\left (2\mathsf A^{\mathsf e_{\alpha}} \mathsf A^{\mathsf e_{\alpha}} -\text{tr}(\mathsf A^{\mathsf e_{\alpha}})\mathsf A^{\mathsf e_{\alpha}}\right )\\
&=-2\text{Ric}^N+\sum_{\alpha=n+1}^{q}\text{tr}(\mathsf A^{\mathsf e_{\alpha}})\mathsf A^{\mathsf e_{\alpha}}=-\text{Ric}^N+\sum_{\alpha=n+1}^{q}\mathsf A^{\mathsf e_{\alpha}} \mathsf A^{\mathsf e_{\alpha}},
\end{aligned}\end{equation*}

Let $\mathsf v, \mathsf v^\top, \mathsf v^{\perp}$ denote a unit vector in $\mathbb R^q$ the tangential projection of $\mathsf v$ onto $N$, and the normal projection of $\mathsf v$ onto $N$ respectively. We can choose an adopted orthonormal basis $\{\mathsf v_{\ell}\}_{{\ell}=1}^q$ in $\mathbb R^q$ such that $\{\mathsf v_{\ell}\}_{{\ell}=1}^n$ is tangent to $N$, and $\{\mathsf v_{\ell}\}_{{\ell}=n+1}^q$ is normal to $N$ at a point in $N$. Denote by $\mathsf f_t^{\mathsf v_{\ell}^\top}$ the flow generated by $\mathsf v_{\ell}^\top$.

\begin{thm}$($An average variation formula for $\Phi$-energy on the target of $u$ which is not necessarily $\Phi$-harmonic $)$\label{T:3.1}
\begin{equation}\label{3.8}
\begin{aligned}
\sum_{\ell=1}^q\frac{d^2}{dt^2}& E_{\Phi}(\mathsf f_t^{\mathsf v_{\ell}^\top} \circ u)_{\big |_{t=0}} =\int_M \sum_{i=1}^m  \left \langle  Q^N \big (du (e_i)\big ),\sum_{j=1}^m \langle  du (e_i) , du (e_j) \rangle du (e_j)\big)  \right \rangle \, dx\\
& \quad + 2 \int_M \sum_{i=1}^m  \sum_{\alpha=n+1}^q \left \langle \mathsf A^{\mathsf v_{\alpha}}\big (du(e_i)\big ), \sum_{j=1}^m  \langle \mathsf A^{\mathsf v_{\alpha}}\big (du(e_i)\big ), du (e_j) \rangle\, du (e_j) \right \rangle\, dx,\\
\end{aligned}
\end{equation}
where $Q^N$ is as in $(3.7^{\prime}).$
\end{thm}

\begin{proof}

 As $\mathsf v_{\ell}$ is parallel in $\mathbb R^q$, we have
\begin{equation}\label{3.9}
\begin{aligned}
\nabla ^u_{e_i}\mathsf v_{\ell}^\top&=\nabla ^N_{du(e_i)}\mathsf v_{\ell}^\top=\left(\nabla ^{\mathbb R ^q}_{du(e_i)}\mathsf v_{\ell}^\top\right)^\top
=\left(\nabla ^{\mathbb R ^q}_{du(e_i)}(\mathsf v_{\ell}-\mathsf v_{\ell}^\bot)\right)^\top\\
&=\mathsf A^{\mathsf v_{\ell}^\bot}(du(e_i))\, ,
\end{aligned}
\end{equation}
Then apply Corollary \ref{C:2.5} to $u_t = \mathsf f_t^{\mathsf v_{\ell}^\top} \circ u$ in which $\mathsf v=\mathsf v_{\ell}^\top$, we have
\begin{equation}\label{3.10}
\begin{aligned}
& \sum_{\ell=1}^q\frac{d^2}{dt^2}E_{\Phi}( \mathsf f_t^{\mathsf v_{\ell}^\top} \circ u)_{\big |_{t=0}}
 =  \int _M \sum_{\ell=1}^q \sum_{i,j=1}^m \langle \nabla ^{u}_{e_i} \mathsf v_{\ell}^\top , du (e_j) \rangle ^2\, dx \\
& \quad + \int _M \sum_{\ell=1}^q \sum_{i,j=1}^m \langle  du (e_i), \nabla ^{u}_{e_j} \mathsf v_{\ell}^\top  \rangle \langle \nabla ^{u}_{e_i} \mathsf v_{\ell}^\top, du (e_j) \rangle\, dx \\
& \quad + \int_M  \sum_{\ell=1}^q\sum_{i,j=1}^m \langle  du (e_i) , du (e_j) \rangle \langle \nabla ^{u}_{e_i} \mathsf v_{\ell}^\top, \nabla ^{u}_{e_j} \mathsf v_{\ell}^\top \rangle \, dx\\
& \quad - \int _M \sum_{\ell=1}^q\bigg \langle \,  \nabla ^{N}_{\mathsf v_{\ell}^\top} \, \mathsf v_{\ell}^\top,\sum_{i,j=1}^m \nabla ^{u}_{e_i} \big (\langle  du (e_i),du (e_j) \rangle du (e_j) \big ) \bigg \rangle\, dx \\
& \quad + \int _M \sum_{\ell=1}^q\sum_{i=1}^m  \bigg \langle R^{N} \big (\mathsf v_{\ell}^\top, du (e_i) \big ) \mathsf v_{\ell}^\top, \sum_{j=1}^m \langle  du (e_i) , du (e_j) \rangle du (e_j) \bigg \rangle \, dx.
\end{aligned}
\end{equation}

In view of \eqref{3.9},  we have the first integrand in \eqref{3.10}

\begin{equation}\label{3.11}
\begin{aligned}
& \quad \sum_{\ell=1}^q \sum_{i,j=1}^m \langle \nabla ^{u}_{e_i} \mathsf v_{\ell}^\top , du (e_j) \rangle ^2 \\
& = \sum_{\ell=1}^q \sum_{i,j=1}^m \langle \mathsf A^{\mathsf v_{\ell}^\bot}\big (du(e_i) \big ), du (e_j) \rangle \langle \mathsf A^{\mathsf v_{\ell}^\bot}\big (du(e_i)\big ) , du (e_j) \rangle \\
& = \sum_{\ell=1}^q \sum_{i=1}^m \left \langle \mathsf A^{\mathsf v_{\ell}^\bot}\big (du(e_i)\big ), \sum_{j=1}^m  \langle \mathsf A^{\mathsf v_{\ell}^\bot}\big (du(e_i)\big ), du (e_j) \rangle\, du (e_j) \right \rangle\\
& = \sum_{\alpha=n+1}^q \sum_{i=1}^m \left \langle \mathsf A^{\mathsf v_{\alpha}}\big (du(e_i)\big ), \sum_{j=1}^m  \langle \mathsf A^{\mathsf v_{\alpha}}\big (du(e_i)\big ), du (e_j) \rangle\, du (e_j) \right \rangle.
\end{aligned}
\end{equation}
The second integrand in \eqref{3.10}, via \eqref{3.9}
\begin{equation}\label{3.12}
\begin{aligned}
&\quad \sum_{\ell=1}^q \sum_{i,j=1}^m \langle  du (e_i), \nabla ^{u}_{e_j} \mathsf v_{\ell}^\top  \rangle \langle \nabla ^{u}_{e_i} \mathsf v_{\ell}^\top, du (e_j) \rangle\\
 & = \sum_{\ell=1}^q \sum_{i,j=1}^m \langle  du (e_i), \mathsf A^{\mathsf v_{\ell}^\bot}\big (du(e_j)\big )  \rangle \langle \mathsf A^{\mathsf v_{\ell}^\bot}\big (du(e_i)\big ) , du (e_j) \rangle
 \end{aligned}
\end{equation}
 \[
\begin{aligned}
 &  = \sum_{\ell=1}^q \sum_{i=1}^m \left \langle \mathsf A^{\mathsf v_{\ell}^\bot}\big (du(e_i)\big ) , \sum_{j=1}^m  \langle  du (e_i), \mathsf A^{\mathsf v_{\ell}^\bot}\big (du(e_j)\big )  \rangle du (e_j) \right \rangle\\
& = \sum_{\alpha=n+1}^q  \sum_{i=1}^m \left \langle \mathsf A^{\mathsf v_{\alpha}}\big (du(e_i)\big ) , \sum_{j=1}^m  \langle  \mathsf A^{\mathsf v_{\alpha}} \big (du (e_i)\big ), du(e_j)  \rangle du (e_j) \right \rangle.
\end{aligned}
\]

The third integrand in \eqref{3.10}
\begin{equation}\label{3.13}
\begin{aligned}
 & \sum_{\ell=1}^q \sum_{i,j=1}^m \langle  du (e_i) , du (e_j) \rangle \langle \nabla ^{u}_{e_i} \mathsf v_{\ell}^\top, \nabla ^{u}_{e_j} \mathsf v_{\ell}^\top \rangle\\
  & =  \sum_{\ell=1}^q \sum_{i,j=1}^m \langle  du (e_i) , du (e_j) \rangle \langle \mathsf A^{\mathsf v_{\ell}^\bot}\big (du(e_i)\big ) , \mathsf A^{\mathsf v_{\ell}^\bot}\big (du(e_j)\big )  \rangle \\
\end{aligned}
\end{equation}
\[
\begin{aligned}
 & =  \sum_{\ell=1}^q \sum_{i,j=1}^m \langle  du (e_i) , du (e_j) \rangle \langle \mathsf A^{\mathsf v_{\ell}^\bot}\mathsf A^{\mathsf v_{\ell}^\bot}\big (du(e_i)\big ) , \big (du(e_j)\big )  \rangle \\
 & =  \sum_{\ell=n+1}^q \sum_{i=1}^m \left \langle \mathsf A^{\mathsf v_{\ell}^\bot}\mathsf A^{\mathsf v_{\ell}^\bot}\big (du(e_i)\big ) , \sum_{j=1}^m \langle  du (e_i) , du (e_j) \rangle \big (du(e_j)\big )  \right \rangle \\
 & =  \sum_{\alpha=n+1}^q \sum_{i=1}^m \left \langle \mathsf A^{\mathsf v_{\alpha}}\mathsf A^{\mathsf v_{\alpha}}\big (du(e_i)\big ) , \sum_{j=1}^m \langle  du (e_i) , du (e_j) \rangle \big (du(e_j)\big )  \right \rangle.
\end{aligned}
\]

Since either $\mathsf v_{\ell}^{\bot} = 0$ or $\mathsf v_{\ell}^{\top} = 0\, $ for each $1 \le \ell \le q\, ,$ \[\sum_{\ell=1}^q \nabla^N_{\mathsf v_{\ell}^\top}\mathsf v_{\ell}^\top=\sum_{\ell=1}^q \big (\nabla^{\mathbb R^q}_{\mathsf v_{\ell}^\top}(\mathsf v_{\ell} - \mathsf v_{\ell}^{\bot})\big )^{\top} = \sum_{\ell=1}^q \big (\nabla^{\mathbb R^q}_{\mathsf v_{\ell}^\top}( - \mathsf v_{\ell}^{\bot})\big )^{\top} = \sum_{\ell=1}^q  \mathsf A^{\mathsf v_{\ell}^{\bot}}(\mathsf v_{\ell}^{\top}) = 0\, ,
\]
we have the forth integrand in \eqref{3.10}
\begin{equation}\label{3.14}
\begin{aligned}
-  \sum_{\ell=1}^q \bigg \langle \,  \nabla ^{N}_{\mathsf v_{\ell}^\top} \, \mathsf v_{\ell}^\top,\sum_{i,j=1}^m \nabla ^{u}_{e_i} \big (\langle  du (e_i),du (e_j) \rangle du (e_j) \big ) \bigg\rangle = 0.
\end{aligned}
\end{equation}

By the Gauss equation, i.e. for every vector field $\mathsf X,\mathsf Y,\mathsf Z,\mathsf W$ on $N$
$$\langle R^N(\mathsf X,\mathsf Y)\mathsf Z,\mathsf W \rangle = \langle R^{\mathbb R^q}(\mathsf X,\mathsf Y)\mathsf Z,\mathsf W\rangle_{\mathbb R^q}  + \langle \mathsf B(\mathsf X,\mathsf W), \mathsf B(\mathsf Y,\mathsf Z)  \rangle_{\mathbb R^q} - \langle \mathsf B(\mathsf X,\mathsf Z), \mathsf B(\mathsf Y,\mathsf W)  \rangle_{\mathbb R^q}\, ,$$ and \eqref{3.3}, the fifth integrand in \eqref{3.10}
\begin{equation}\label{3.15}
\begin{aligned}
&\sum_{\ell=1}^q\sum_{i=1}^m  \left \langle R^{N} \big (\mathsf v_{\ell}^\top, du (e_i) \big ) \mathsf v_{\ell}^\top, \sum_{j=1}^m \langle  du (e_i) , du (e_j) \rangle du (e_j) \right \rangle \\
&=  \sum_{\ell=1}^q\sum_{i=1}^m  \left \langle \mathsf  B\big (\mathsf v_{\ell}^\top,\sum_{j=1}^m \langle  du (e_i) , du (e_j) \rangle du (e_j)\big ), \mathsf B\big (\mathsf v_{\ell}^\top, du (e_i)\big ) \right \rangle_{\mathbb R^q} \\
&\quad - \sum_{\ell=1}^q\sum_{i=1}^m \left \langle \mathsf B(\mathsf v_{\ell}^\top, \mathsf v_{\ell}^\top), \mathsf B \big (du (e_i),\sum_{j=1}^m \langle  du (e_i) , du (e_j) \rangle du (e_j)\big) \right \rangle_{\mathbb R^q}\\
&=  \sum_{\ell=1}^q\sum_{i=1}^m \sum_{{\alpha}=n+1}^q  \left \langle \mathsf  B\big (\mathsf v_{\ell}^\top,\sum_{j=1}^m \langle  du (e_i) , du (e_j) \rangle du (e_j)\big ), \mathsf v_{\alpha} \right \rangle_{\mathbb R^q}  \cdot \left \langle \mathsf B\big (\mathsf v_{\ell}^\top, du (e_i)\big ), \mathsf v_{\alpha} \right \rangle_{\mathbb R^q} \\
\end{aligned}
\end{equation}
\[
\begin{aligned}
&\quad - \sum_{\ell=1}^q\sum_{i=1}^m \sum_{{\alpha}=n+1}^q \langle \mathsf  B(\mathsf v_{\ell}^\top, \mathsf v_{\ell}^\top), \mathsf v_{\alpha} \rangle_{\mathbb R^q}\,   \cdot  \left \langle \mathsf  B\big (du (e_i),\sum_{j=1}^m \langle  du (e_i) , du (e_j) \rangle du (e_j)\big), \mathsf v_{\alpha}  \right \rangle_{\mathbb R^q}\\
&=  \sum_{\ell=1}^q\sum_{i=1}^m \sum_{{\alpha}=n+1}^q  \left \langle \mathsf A^{\mathsf v_{\alpha} }\big (\sum_{j=1}^m \langle  du (e_i) , du (e_j) \rangle du (e_j)\big ), \mathsf v_{\ell}^\top \right \rangle \cdot \left \langle \mathsf A^{\mathsf v_{\alpha}}\big ( du (e_i)\big ),\mathsf v_{\ell}^\top  \right \rangle \\
&\quad - \sum_{\ell=1}^q\sum_{i=1}^m \sum_{{\alpha}=n+1}^q \langle \mathsf A^{\mathsf v_{\alpha} } (\mathsf v_{\ell}^\top), \mathsf v_{\ell}^\top \rangle\,   \cdot  \left \langle  \mathsf A^{\mathsf v_{\alpha} } \big (du (e_i)\big ),\sum_{j=1}^m \langle  du (e_i) , du (e_j) \rangle du (e_j)\big)  \right \rangle\\
&=  \sum_{i=1}^m \sum_{{\alpha}=n+1}^q  \left \langle \mathsf A^{\mathsf v_{\alpha} }\big (\sum_{j=1}^m \langle  du (e_i) , du (e_j) \rangle du (e_j)\big ),  \mathsf A^{\mathsf v_{\alpha}}\big ( du (e_i)\big )  \right \rangle \\
&\quad - \sum_{i=1}^m \sum_{{\alpha}=n+1}^q \text{trace} \, (\mathsf A^{\mathsf v_{\alpha} } ) \left \langle  \mathsf A^{\mathsf v_{\alpha} } \big (du (e_i)\big ),\sum_{j=1}^m \langle  du (e_i) , du (e_j) \rangle du (e_j)  \right \rangle\\
&=  \sum_{i=1}^m \sum_{{\alpha}=n+1}^q  \left \langle \mathsf A^{\mathsf v_{\alpha} }\mathsf A^{\mathsf v_{\alpha}}\big ( du (e_i)\big ), \sum_{j=1}^m \langle  du (e_i) , du (e_j) \rangle du (e_j)    \right \rangle \\
&\quad - \sum_{i=1}^m \sum_{{\alpha}=n+1}^q \text{trace} \, (\mathsf A^{\mathsf v_{\alpha} } ) \left \langle  \mathsf A^{\mathsf v_{\alpha} } \big (du (e_i)\big ),\sum_{j=1}^m \langle  du (e_i) , du (e_j) \rangle du (e_j)\big)  \right \rangle.
\end{aligned}
\]
Substituting \eqref{3.11},\eqref{3.12},\eqref{3.13},\eqref{3.14} and \eqref{3.15} into \eqref{3.10}, we obtain the desired \eqref{3.8}.
\end{proof}

Similarly, we can isometrically immerse $M$ into $\mathbb R^q$. Let $\{v_{\ell}^\top\}$ be the tangential projection of an orthonormal frame field $\{v_{\ell}\}_{\ell=1}^q$ in $\mathbb R^q$ onto $M$. Denote by $f_t^{v_{\ell}^\top}: M \to M$ the flow generated by $v_{\ell}^\top$, apply Corollary \ref{C:2.6} with $u_t=u\circ f_t^{v_{\ell}^{\top}}$ and $u_0=u$. For convenience, we choose $\{v_1,\cdots,v_m\}=\{e_1,\cdots,e_m\}$ to be tangential to $M$, $\{v_{m+1},\cdots,v_q\}=\{e_{m+1},\cdots,e_{q}\}$ to be normal to $M$, and $\nabla ^{\Psi}e_i = 0$ at a point in $M$. We have

\begin{thm}$($An average variation formula for $\Phi$-energy on the domain of a $\Phi$-harmonic map $u$ $)$\label{41}
\begin{equation}\label{3.16}
\begin{aligned}
&\sum_{\ell=1}^q\frac{d^2}{dt^2}E_{\Phi}( u \circ f_t^{v_{\ell}^\top})_{\big |_{t=0}}\\
& =\int_M \sum_{i=1}^m  \left \langle  du \big ( Q^M (e_i)\big ),\sum_{j=1}^m \langle  du (e_i) , du (e_j) \rangle du (e_j)\big)  \right \rangle \, dx\\
& \quad + 2 \int_M \sum_{i=1}^m  \sum_{\alpha=m+1}^q \left \langle du\big (A^{e_{\alpha}}(e_i)\big ), \sum_{j=1}^m  \langle du(e_i), du\big (  A^{e_{\alpha}}(e_j)\big ) \rangle\, du (e_j) \right \rangle\, dx,\\
\end{aligned}
\end{equation}
where $Q^M$ is as in \eqref{3.7}.
\end{thm}

\begin{proof} 
Applying \eqref{2.12} in which $v$ is replaced by $du(v_{\ell}^\top)\, ,$ we have
\begin{equation}
\begin{aligned}
& \sum_{\ell=1}^q\frac{d^2}{dt^2}E_{\Phi}( u \circ f_t^{v_{\ell}^\top})_{\big |_{t=0}}  \\
 & =  \int _M \sum_{\ell=1}^q \sum_{i,j=1}^m \langle \nabla ^{u}_{e_i} du(v_{\ell}^\top) , du (e_j) \rangle ^2\, dx\\
  &\quad+ \int _M \sum_{\ell=1}^q \sum_{i,j=1}^m \langle  du (e_i), \nabla ^{u}_{e_j} du (v_{\ell}^\top)  \rangle \langle \nabla ^{u}_{e_i} du (v_{\ell}^\top), du (e_j) \rangle\, dx \\
& \quad + \int_M  \sum_{\ell=1}^q\sum_{i,j=1}^m \langle  du (e_i) , du (e_j) \rangle \langle \nabla ^{u}_{e_i} du (v_{\ell}^\top), \nabla ^{u}_{e_j} du (v_{\ell}^\top) \rangle \, dx\\
& \quad + \int _M \sum_{\ell=1}^q\sum_{i=1}^m  \bigg \langle R^{N} \big (du (v_{\ell}^\top), du (e_i) \big ) du (v_{\ell}^\top), \sum_{j=1}^m \langle  du (e_i) , du (e_j) \rangle du (e_j) \bigg \rangle \, dx.
\end{aligned}\label{3.17}
\end{equation}

Since $v_{\ell}^\top=v_{\ell}-v_{\ell}^\bot$ and $v_{\ell}$ are parallel in $\mathbb R^q$, we have
\begin{equation}
\begin{aligned}
\nabla ^u_{e_i} du(v_{\ell}^\top)&=(\nabla ^u_{e_i} du)(v_{\ell}^\top) + du (\nabla ^M_{e_i}v_{\ell}^\top)= (\nabla ^u_{e_i} du)(v_{\ell}^\top) + du \bigg (\left(\nabla ^{\mathbb R ^q}_{e_i}(v_{\ell}-v_{\ell}^\bot)\right)^\top\bigg )\\
&= (\nabla ^u_{e_i} du)(v_{\ell}^\top) + du \left(A^{v_{\ell}^\bot}(e_i)\right )\, .
\end{aligned}\label{3.18}
\end{equation}

As for each $1 \le \ell \le q\, ,$ either $v_{\ell}^{\bot} = 0$ or $v_{\ell}^{\top} = 0\, ,$
\begin{equation}\label{3.19}
 \langle (\nabla ^u_{e_i} du)(v_{\ell}^\top) , du (e_j) \rangle  \langle du \big (A^{v_{\ell}^\bot}(e_i)\big ) , du (e_j) \rangle = 0\, .
\end{equation}

In view of \eqref{3.18} and \eqref{3.19},  we have the first integrand in \eqref{3.17}

\begin{equation}\label{3.20}
\begin{aligned}
& \sum_{\ell=1}^q \sum_{i,j=1}^m \langle \nabla ^{u}_{e_i} du(v_{\ell}^\top) , du (e_j) \rangle ^2\\
& = \sum_{\ell=1}^q \sum_{i,j=1}^m \bigg ( \langle (\nabla ^u_{e_i} du)(v_{\ell}^\top) , du (e_j) \rangle + \langle du \left(A^{v_{\ell}^\bot}(e_i)\right ) , du (e_j) \rangle\bigg )^2\\
& = \sum_{\ell=1}^q \sum_{i,j=1}^m\langle (\nabla ^u_{e_i} du)(v_{\ell}^\top) , du (e_j) \rangle^2 \\
&\quad+ 2 \langle (\nabla ^u_{e_i} du)(v_{\ell}^\top) , du (e_j) \rangle  \langle du \big (A^{v_{\ell}^\bot}(e_i)\big ) , du (e_j) \rangle \\
& \quad + \langle du \big (A^{v_{\ell}^\bot}(e_i)\big ) , du (e_j) \rangle^2 \\
 & = \sum_{\ell=1}^q \sum_{i,j=1}^m\langle (\nabla ^u_{v_{\ell}^\top} du)(e_i) , du (e_j) \rangle^2 \\
 &\quad+ \sum_{\ell=1}^q \sum_{i,j=1}^m \langle du\big (A^{v_{\ell}^\bot}(e_i) \big ), du (e_j) \rangle \langle du\big (A^{v_{\ell}^\bot}(e_i)\big ) , du (e_j) \rangle \\
 & = \sum_{i,j,k=1}^m\langle (\nabla ^u_{e_k} du)(e_i) , du (e_j) \rangle^2 \\
 &\quad+ \sum_{\alpha=m+1}^q \sum_{i,j=1}^m \langle du\big (A^{e_{\alpha}}(e_i) \big ), du (e_j) \rangle \langle du\big (A^{e_{\alpha}}(e_i)\big ) , du (e_j) \rangle
 \end{aligned}
\end{equation}
 \begin{equation*}
\begin{aligned}
 & = \sum_{i,j,k=1}^m\langle (\nabla ^u_{e_k} du)(e_i) , du (e_j) \rangle^2\\
  &\quad+ \sum_{\alpha=m+1}^q \sum_{i,j,k=1}^m \langle du\big (A^{e_{\alpha}}(e_i)\big ) , du (e_j) \rangle \left \langle du\big (\langle A^{e_{\alpha}}(e_i), e_k \rangle e_k \big ), du (e_j) \right \rangle \\
 & = \sum_{i,j,k=1}^m\langle (\nabla ^u_{e_k} du)(e_i) , du (e_j) \rangle^2 \\
 &\quad+ \sum_{\alpha=m+1}^q \sum_{i,j,k=1}^m \left  \langle du\bigg (A^{e_{\alpha}}\big (\langle A^{e_{\alpha}}(e_k), e_i \rangle e_i\big )\bigg ) , du (e_j) \right \rangle \langle du\big (e_k \big ), du (e_j) \rangle\\
 & = \sum_{i,j,k=1}^m\langle (\nabla ^u_{e_k} du)(e_i) , du (e_j) \rangle^2 \\
 &\quad+ \sum_{\alpha=m+1}^q \sum_{j,k=1}^m \langle du\bigg  (A^{e_{\alpha}}\big (A^{e_{\alpha}}(e_k)\big )\bigg ) , du (e_j) \rangle \langle du\big (e_k \big ), du (e_j) \rangle \\
 & = \sum_{i,j,k=1}^m\langle (\nabla ^u_{e_k} du)(e_i) , du (e_j) \rangle^2 \\
 &\quad+ \sum_{\alpha=m+1}^q \sum_{i=1}^m \left \langle du\bigg  (A^{e_{\alpha}}\big (A^{e_{\alpha}}(e_i)\big )\bigg ) , \sum_{j=1}^m \langle du\big (e_i \big ), du (e_j) \rangle du (e_j) \right \rangle.  \\
\end{aligned}
\end{equation*}
Analogous to \eqref{3.19} for each $1 \le \ell \le q\, ,$
\begin{equation}\label{3.21}
 \langle du \big (A^{v_{\ell}^\bot}(e_i)\big ) , du (e_j) \rangle \langle (\nabla ^u_{e_j} du)(v_{\ell}^\top) , du (e_i) \rangle   = 0\, .
\end{equation}

By \eqref{3.18} and \eqref{3.21} the second integrand in \eqref{3.17}

\begin{equation}\label{3.22}
\begin{aligned}
&\quad  \sum_{\ell=1}^q \sum_{i,j=1}^m \langle  du (e_i), \nabla ^{u}_{e_j} du (v_{\ell}^\top)  \rangle \langle \nabla ^{u}_{e_i} du (v_{\ell}^\top), du (e_j) \rangle \\
& = \sum_{\ell=1}^q \sum_{i,j=1}^m \bigg (\langle du (e_i), (\nabla ^u_{e_j} du)(v_{\ell}^\top)  \rangle + \langle du (e_i), du\left(A^{v_{\ell}^\bot}(e_j)\right )\rangle\bigg )\\
& \quad \cdot \bigg (\langle (\nabla ^u_{e_i} du)(v_{\ell}^\top),du (e_j)\rangle + \langle du\left(A^{v_{\ell}^\bot}(e_i)\right ),du (e_j)\rangle\bigg )\\
& = \sum_{\ell=1}^q \sum_{ij=1}^m \bigg ( \langle du (e_i), (\nabla ^u_{e_j} du)(v_{\ell}^\top)  \rangle  \langle (\nabla ^u_{e_i} du)(v_{\ell}^\top) , du (e_j)\rangle \\
\end{aligned}
\end{equation}
\begin{equation*}
\begin{aligned}
& \quad + \langle du (e_i), du\big (A^{v_{\ell}^\bot}(e_j)\big )\rangle \langle du\big(A^{v_{\ell}^\bot}(e_i)\big ), du (e_j)\rangle\bigg )\\
& = \sum_{i,j,k=1}^m  \langle du (e_i), (\nabla ^u_{e_k} du) (e_{j})  \rangle  \langle (\nabla ^u_{e_k} du)(e_i) , du (e_j) \rangle \\
& \quad + \sum_
{\alpha=m+1}^q \sum_{i,j=1}^m \langle du (e_i), du \big ( A^{e_{\alpha}}(e_j) \big  )\rangle  \langle du \big  (A^{e_{\alpha}}(e_i)\big  ) , du (e_j) \rangle\\
& = \sum_{i,j,k=1}^m  \langle du (e_i), (\nabla ^u_{e_k} du) (e_{j})  \rangle  \langle (\nabla ^u_{e_k} du)(e_i) , du (e_j) \rangle \\
&\quad + \sum_
{\alpha=m+1}^q \sum_{i=1}^m  \left \langle du\big (A^{e_{\alpha}}(e_i)\big ), \sum_{j=1}^m  \langle du(e_i), du\big (  A^{e_{\alpha}}(e_j)\big ) \rangle\, du (e_j) \right \rangle.
\end{aligned}
\end{equation*}

By \eqref{3.18} and \eqref{3.21} the third integrand in \eqref{3.17}

\begin{equation}\label{3.23}
\begin{aligned}
&\sum_{\ell=1}^q\sum_{i,j=1}^m \langle  du (e_i) , du (e_j) \rangle \langle \nabla ^{u}_{e_i} du (v_{\ell}^\top), \nabla ^{u}_{e_j} du (v_{\ell}^\top) \rangle\\
= &\sum_{\ell=1}^q\sum_{i,j=1}^m \langle  du (e_i) , du (e_j) \rangle \langle (\nabla ^u_{e_i} du)(v_{\ell}^\top) + du \left(A^{v_{\ell}^\bot}(e_i)\right ), (\nabla ^u_{e_j} du)(v_{\ell}^\top) + du \left(A^{v_{\ell}^\bot}(e_j)\right )\rangle\\
= &\sum_{\ell=1}^q\sum_{i,j=1}^m \langle  du (e_i) , du (e_j) \rangle \langle (\nabla ^u_{e_i} du)(v_{\ell}^\top) , (\nabla ^u_{e_j} du)(v_{\ell}^\top)\rangle \\
& + \sum_{\ell=1}^q\sum_{i,j=1}^m \langle  du (e_i) , du (e_j) \rangle \langle du \big (A^{v_{\ell}^\bot}(e_i)\big ), du \big (A^{v_{\ell}^\bot}(e_j)\big )\rangle\\
= &\sum_{i,j,k=1}^m \langle  du (e_i) , du (e_j) \rangle \langle (\nabla ^u_{e_k} du)(e_i) , (\nabla ^u_{e_k} du)(e_j)\rangle \\
& + \sum_{\alpha=m+1}^q\sum_{i,j,k=1}^m \langle  du (e_i) , du (e_j) \rangle \left \langle du \big (A^{e_{\alpha}}(e_i)\big ), du \big (\langle  A^{e_{\alpha}}(e_j), e_k\rangle e_k \big ) \right \rangle\\
= &\sum_{i,j,k=1}^m \langle  du (e_i) , du (e_j) \rangle \langle (\nabla ^u_{e_k} du)(e_i) , (\nabla ^u_{e_k} du)(e_j)\rangle
\end{aligned}
\end{equation}
\begin{equation*}
\begin{aligned}
& + \sum_{\alpha=m+1}^q\sum_{i,j,k=1}^m \langle  du (e_i) , du \big (\langle  A^{e_{\alpha}}(e_k), e_j\rangle e_j\big ) \rangle \left \langle du \big (A^{e_{\alpha}}(e_i)\big ), du ( e_k ) \right \rangle\\
= &\sum_{i,j,k=1}^m \langle  du(e_i), du (e_j) \rangle \langle (\nabla ^u_{e_k} du)(e_i) , (\nabla ^u_{e_k} du)(e_j)\rangle \\
& +\sum_{\alpha=m +1}^q\sum_{i,j=1}^m \langle  du(e_i), du \big ( A^{e_{\alpha}}(e_j)\big ) \rangle \left \langle du \big (A^{e_{\alpha}}(e_i)\big ), du(e_j) \right \rangle\\
= &\sum_{i,j,k=1}^m \langle  du (e_i) , du (e_j) \rangle \langle (\nabla ^u_{e_k} du)(e_i) , (\nabla ^u_{e_k} du)(e_j)\rangle \\
&+ \sum_
{\alpha=m+1}^q \sum_{i=1}^m  \left \langle du\big (A^{e_{\alpha}}(e_i)\big ), \sum_{j=1}^m  \langle du(e_i), du\big (  A^{e_{\alpha}}(e_j)\big ) \rangle\, du (e_j) \right \rangle.
\end{aligned}
\end{equation*}

By the Weitzenb\"{o}ck formula for $1$-form [EL, Proposition 1.34, p.13]
\begin{equation}\label{3.24}
\sum _{i=1}^{m} R^{N} \big (du (v_{\ell}^\top), du (e_i) \big ) du (e_i) =  du (\text{Ric} ^M(v_{\ell}^\top)) - \sum _{i=1}^{m}(\nabla_{e_i}\nabla_{e_i} du) (v_{\ell}^\top) + \triangle (du) (v_{\ell}^\top),
\end{equation}
where $\triangle$ is the Hodge Laplacian given by $\triangle = - (d^{\ast} d + d d^{\ast})\, $ in which $d$ is the exterior differential operator, $d^{\ast}$ is the codifferental operator, and $d^{\ast} du = - \sum _{i=1}^m (\nabla^u _{e_i} du) (e_i)\, $.

By \eqref{3.24},  the forth integrand in \eqref{3.17}
\begin{equation}\label{3.25}
\begin{aligned}
&\sum_{\ell=1}^q\sum_{i=1}^m  \left \langle R^{N} \big (du (v_{\ell}^\top), du (e_i) \big ) du (v_{\ell}^\top), \sum_{j=1}^m \langle  du (e_i) , du (e_j) \rangle du (e_j) \right \rangle \\  
= &\sum_{i,j,k=1}^m  - \langle R^{N} \big (du (e_i), du (e_k) \big ) du (e_k),  du (e_j) \rangle \langle  du (e_i) , du (e_j) \rangle  \\
= &\sum_{i,j,k=1}^m \langle du \big ( - \text{Ric} ^M(e_i) \big )+ (\nabla_{e_k}\nabla_{e_k} du) (e_i) - \triangle (du) (e_i),  du (e_j) \rangle \langle  du (e_i) , du (e_j) \rangle  \\
= &\sum_{i,=1}^m \left \langle du \big ( - \text{Ric} ^M(e_i)\big ), \sum_{j=1}^m \langle  du (e_i) , du (e_j) \rangle du (e_j) \right \rangle  \\
& \quad + \sum_{i=1}^m \left \langle - \triangle (du) (e_i), \sum_{j=1}^m \langle  du (e_i) , du (e_j) \rangle du (e_j) \right \rangle   \\
\end{aligned}
\end{equation}
\[
\begin{aligned}
&\quad+  \sum_{i,,j,k=1}^m e_k \big (\langle \nabla_{e_k} du (e_i), du (e_j) \rangle \langle  du (e_i) , du (e_j) \rangle \big )\\
 &\quad- \sum_{i,j,k=1}^m \langle \nabla_{e_k} du (e_i), \nabla_{e_k} du (e_j)\rangle \langle  du (e_i) , du (e_j) \rangle   \\
&\quad- \sum_{i,j,k=1}^m  \langle \nabla_{e_k} du (e_i),  du (e_j)\rangle ^2 \\
  &\quad- \sum_{i,j,k=1}^m  \langle \nabla_{e_k} du (e_i),  du (e_j)\rangle \langle  du (e_i) , \nabla_{e_k} du (e_j) \rangle .
\end{aligned}
\]
Since $u$ is $\Phi$-harmonic, applying  $d(du) = 0\, ,$ and Corollary \ref{C:2.2}, we have
\begin{equation}\label{3.26}
\begin{aligned}
& \int _M \sum_{i=1}^m \left \langle \triangle (du) (e_i), \sum_{j=1}^m \langle  du (e_i) , du (e_j) \rangle du (e_j) \right \rangle\, dx \\
= & \int _M \sum_{i=1}^m \left \langle -(d^{\ast}d + dd^{\ast}) (du) (e_i), \sum_{j=1}^m \langle  du (e_i) , du (e_j) \rangle du (e_j) \right \rangle\, dx\\
= & \int _M \sum_{i=1}^m \left \langle -d^{\ast} (du) (e_i), d^{\ast}\sum_{j=1}^m \langle  du (e_i) , du (e_j) \rangle du (e_j) \right \rangle\, dx\\
= & \int _M  \left \langle -d^{\ast} (du), d^{\ast}\sum_{j=1}^m \langle  du , du (e_j) \rangle du (e_j) \right \rangle\, dx\\
= & \int _M  \left \langle -d^{\ast} (du), \sum_{i=1}^m  \nabla^u _{e_i} \bigg (\sum_{j=1}^m \langle  du , du (e_j) \rangle du (e_j) \bigg ) (e_i)\right \rangle\, dx\\
= & \int _M  \left \langle -d^{\ast} (du), \sum_{i=1}^m \nabla^u _{e_i} \bigg (\sum_{j=1}^m \langle  du  (e_i) , du (e_j) \rangle du (e_j) \bigg )\right \rangle\, dx\\
= & 0.
\end{aligned}
\end{equation}
Furthermore, by \eqref{2.50}
\begin{equation}\label{3.27}
\begin{aligned}
&\sum_{i,,j,k=1}^m e_k \big (\langle \nabla_{e_k} du (e_i), du (e_j) \rangle \langle  du (e_i) , du (e_j) \rangle \big ) \\
\end{aligned}
\end{equation}
\begin{equation*}
\begin{aligned}
& =  \sum_{i,j,k=1}^m  \text{div}_M \big (\langle \nabla_{e_k} du (e_i), du (e_j) \rangle \langle  du (e_i) , du (e_j) \rangle  e_k \big ).
\end{aligned}
\end{equation*}
Substituting \eqref{3.20},\eqref{3.22},\eqref{3.23} and \eqref{3.26} into \eqref{3.17}, and  applying \eqref{3.26} and \eqref{3.27},
we obtain via \eqref{3.7} the desired \eqref{3.16}.
\end{proof}

\section{Average variational method Part II: $\Phi$-SSU and $\Phi$-SU manifolds }

In this section, using the technique in \cite{WY}, we write the average variation formulas in orthogonal notation in terms of the differential matrix $(u_{i\alpha})$ of $u$. This enables us to make estimates on the variation formulas from which we find $\Phi$-SSU manifolds. Applying an average method, we prove that $\Phi$-SSU manifolds are $\Phi$-SU; i.e.,  Theorem $1.1 (a),(b),(c)$, and $(d)$ hold.

\begin{lem}[An Estimate on the Average Variation Formula \eqref{3.8} on the Target of $u$]  Let $\mathsf C_{\alpha}$ be as in \eqref{4.4}. Then
\begin{equation}\label{4.1}
\begin{aligned}
&\sum_{\ell=1}^q\frac{d^2}{dt^2}E_{\Phi}( \mathsf f_t^{\mathsf v_{\ell}^\top} \circ u)_{\big |_{t=0}}\\
& \le \int_M \sum_{\alpha =1}^n \mathsf C_{\alpha}^2 \sum_{\beta =1}^n \bigg (4 \langle  \mathsf B (\mathsf e_{\alpha} , \mathsf e_{\beta} ), \mathsf B (\mathsf e_{\alpha}  , \mathsf e_{\beta} ) \rangle - \langle \mathsf B (\mathsf e_{\alpha}  , \mathsf e_{\alpha} ), \mathsf B (\mathsf e_{\beta}  , \mathsf e_{\beta} \rangle \bigg )\, dx.\\
\end{aligned}
\end{equation}\label{L:4.1}
\end{lem}
\begin{proof}[Proof of Lemma 4.1]
Denote $(u_{i \alpha})\, 1 \le i \le m, 1 \le \alpha \le n$ the differential matrix of $u: M \to N$ relative to local orthonormal bases
$\{e_1,\cdots,e_m\}$ in $M$ and $\{\mathsf e_1,\cdots,\mathsf e_n\}$ in $N$. That is, \begin{equation}\label{4.2}du(e_i)=\sum_{\alpha=1}^nu_{i\alpha} \mathsf e_{\alpha}\quad \text{for}\quad 1 \le i \le m\, .\end{equation} Let $(u_{i \alpha})^{\text T}$ be the transpose of $(u_{i \alpha})\, .$
Then the product matrix $(u_{i \alpha})^{\text T} \cdot (u_{i \alpha})$ is an $n\times n$ symmetric  matrix with the $\alpha\, ,$ $\beta$ entry over the field of real numbers given by
\begin{eqnarray}\label{4.3}
\left( \sum_{i=1}^mu_{i\alpha}u_{i\beta}\right).
\end{eqnarray}
Without ambiguity, we use the same notations for local orthonormal bases $\{e_1,\cdots,e_m\}$ in $M$ and $\{\mathsf e_1,\cdots,\mathsf e_n\}$ in $N$ so that the product matrix is diagonalizable. That is,
\begin{equation}\label{4.4}
\sum_{i=1}^mu_{i \alpha}u_{i \beta}=
\begin{cases}
\begin{aligned}
0\quad &\text{if}\quad  \beta \neq \alpha\\
\mathsf C_{\alpha}  \quad &\text{if}\quad  \beta = \alpha
\end{aligned}
\end{cases}
\end{equation}
for some $\mathsf C_{\alpha} \ge 0\, .$
Then by \eqref{4.4} and \eqref{1.1} we have

\begin{equation}\label{4.100}
2 e(u) = \sum_{i=1}^m\langle du(e_i),du(e_i)\rangle =\sum_{\alpha =1}^n \mathsf C_{\alpha},
\end{equation}

\begin{equation}\label{4.5}
\begin{aligned}
4 e_{\Phi}(u)&=\sum_{ij=1}^m\langle du(e_i),du(e_j)\rangle^2 =\sum_{ij=1}^m\sum_{\alpha\beta\gamma\delta=1}^n\langle u_{i\alpha} \mathsf e_{\alpha} ,u_{j\beta} \mathsf e_{\beta}\rangle \langle u_{i\gamma} \mathsf e_{\gamma} ,u_{j\delta} \mathsf e_{\delta}\rangle  \\
& =\sum_{ij=1}^m\sum_{\alpha\gamma=1}^n u_{i\alpha}u_{j\alpha}u_{i\gamma}u_{j\gamma} =\sum_{\alpha\gamma=1}^n (\sum_{i=1}^m u_{i\alpha}u_{i\gamma}\sum_{j=1}^mu_{j\alpha}u_{j\gamma}) =\sum_{\alpha=1}^n\mathsf C_{\alpha}^2
\end{aligned}
\end{equation}
and via \eqref{3.3},
\begin{equation}\label{4.6}
\begin{aligned}
\mathsf A^{\mathsf e_{\nu}} \big (du(e_i)\big )&=\sum_{\alpha = 1}^n u_{i\alpha} \mathsf A^{\mathsf e_{\nu}} (\mathsf e_{\alpha}  ) = \sum_{\alpha\beta = 1}^n u_{i\alpha} \langle \mathsf A^{\mathsf e_{\nu}} (\mathsf e_{\alpha}  ), \mathsf e_{\beta} \rangle \mathsf e_{\beta} \\
& = \sum_{\alpha\beta = 1}^n u_{i\alpha} \langle \mathsf B (\mathsf e_{\alpha}  , \mathsf e_{\beta} ), \mathsf e_{\nu} \rangle \mathsf e_{\beta}.\end{aligned}
\end{equation}

In view of \eqref{4.6}, \eqref{3.3} and \eqref{4.4}, we have

\begin{equation}\label{4.7}
\begin{aligned}
&\sum_{ij=1}^m\sum_{\nu= n+1}^q 2 \langle \mathsf A^{\mathsf e_{\nu}} \mathsf A^{\mathsf e_{\nu}} \big (du(e_i)\big ), du(e_j) \rangle \langle du(e_i), du(e_j) \rangle\\
&=\sum_{ij=1}^m\sum_{\nu= n+1}^q\sum_{\alpha\beta\gamma\delta\tau\mu = 1}^n 2 u_{i\alpha} \langle \mathsf B (\mathsf e_{\alpha}  , \mathsf e_{\beta} ), \mathsf e_{\nu} \rangle \langle \mathsf A^{\mathsf e_{\nu}} (\mathsf e_{\beta}), \mathsf e_{\gamma} \rangle \langle \mathsf e_{\gamma }, u_{j\delta}\mathsf e_{\delta}\rangle \langle u_{i\tau}\mathsf e_{\tau }, u_{j\mu}\mathsf e_{\mu}\rangle\\
& = \sum_{\alpha\beta\gamma\tau = 1}^n 2 \big (\sum_{i=1}^m u_{i\alpha} u_{i\tau}\sum_{j=1}^mu_{j\gamma}u_{j\tau}\sum_{\nu= n+1}^q \langle \mathsf B (\mathsf e_{\alpha}  , \mathsf e_{\beta} ), \mathsf e_{\nu} \rangle  \langle \mathsf B (\mathsf e_{\gamma}  , \mathsf e_{\beta}), \mathsf e_{\nu} \rangle \big ) \\
& = \sum_{\alpha\beta = 1}^n 2 \mathsf C_{\alpha}^2 \langle \mathsf B (\mathsf e_{\alpha}  , \mathsf e_{\beta} ), \mathsf B (\mathsf e_{\alpha}  , \mathsf e_{\beta} ) \rangle.
 \end{aligned}
\end{equation}

Similarly, by \eqref{4.6}, \eqref{3.3} and \eqref{4.4}, we have

\begin{equation}\label{4.8}
\begin{aligned}
& \sum_{ij=1}^m\sum_{\nu= n+1}^q -\langle \text{tr} (\mathsf A^{\mathsf e_{\nu}} )\mathsf A^{\mathsf e_{\nu}} \big (du(e_i)\big ), du(e_j) \rangle \langle du(e_i), du(e_j) \rangle\\
&= \sum_{ij=1}^m\sum_{\nu= n+1}^q\sum_{\alpha\beta\gamma\delta\tau\mu = 1}^n -\langle \mathsf A^{\mathsf e_{\nu}}(\mathsf e_{\gamma} ), \mathsf e_{\gamma} \rangle u_{i\alpha} \langle \mathsf B (\mathsf e_{\alpha}  , \mathsf e_{\beta} ), \mathsf e_{\nu} \rangle \langle \mathsf e_{\beta}, u_{j\delta}\mathsf e_{\delta}\rangle \langle u_{i\tau}\mathsf e_{\tau }, u_{j\mu}\mathsf e_{\mu}\rangle\\
& = \sum_{\alpha\beta\gamma\tau = 1}^n -\big (\sum_{i=1}^m u_{i\alpha} u_{i\tau}\sum_{j=1}^mu_{j\beta}u_{j\tau}\sum_{\nu= n+1}^q \langle \mathsf B (\mathsf e_{\gamma}  , \mathsf e_{\gamma} ), \mathsf e_{\nu} \rangle  \langle \mathsf B (\mathsf e_{\alpha}  , \mathsf e_{\beta} ), \mathsf e_{\nu} \rangle \big ) \\
& = \sum_{\alpha\gamma = 1}^n -\mathsf C_{\alpha}^2 \langle \mathsf B (\mathsf e_{\gamma}  , \mathsf e_{\gamma} ), \mathsf B (\mathsf e_{\alpha}  , \mathsf e_{\alpha}) \rangle \\
& = \sum_{\alpha\beta = 1}^n - \mathsf C_{\alpha}^2 \langle \mathsf B (\mathsf e_{\alpha}  , \mathsf e_{\alpha} ), \mathsf B (\mathsf e_{\beta}  , \mathsf e_{\beta}) \rangle.
 \end{aligned}
\end{equation}

In view of \eqref{4.6}, \eqref{3.3} and \eqref{4.4}, we have

\begin{equation}\label{4.9}
\begin{aligned}
&\sum_{ij=1}^m\sum_{\nu= n+1}^q 2 \langle  \mathsf A^{\mathsf e_{\nu}} \big (du(e_i)\big ), du(e_j) \rangle \langle \mathsf A^{\mathsf e_{\nu}} \big (du(e_i)\big ), du(e_j) \rangle\\
&=  \sum_{ij=1}^m\sum_{\nu= n+1}^q\sum_{\alpha\beta\gamma\delta\tau\mu = 1}^n  2 u_{i\alpha} \langle \mathsf B (\mathsf e_{\alpha}  , \mathsf e_{\beta} ), \mathsf e_{\nu} \rangle \langle \mathsf e_{\beta}, u_{j\gamma}\mathsf e_{\gamma}\rangle u_{i\delta} \langle \mathsf B (\mathsf e_{\delta}  , \mathsf e_{\tau} ), \mathsf e_{\nu} \rangle \langle \mathsf e_{\tau}, u_{j\mu}\mathsf e_{\mu}\rangle\\
& = \sum_{\alpha\beta\delta\tau = 1}^n 2 \big (\sum_{i=1}^m u_{i\alpha} u_{i\delta}\sum_{j=1}^mu_{j\beta}u_{j\tau}\sum_{\nu= n+1}^q \langle \mathsf B (\mathsf e_{\alpha}  , \mathsf e_{\beta} ), \mathsf e_{\nu} \rangle  \langle \mathsf B (\mathsf e_{\delta}  , \mathsf e_{\tau} ), \mathsf e_{\nu} \rangle \big ) \\
& = \sum_{\alpha\beta = 1}^n 2 \mathsf C_{\alpha} \mathsf C_{\beta}\langle \mathsf B (\mathsf e_{\alpha}  , \mathsf e_{\beta} ), \mathsf B (\mathsf e_{\alpha}  , \mathsf e_{\beta} ) \rangle \le \sum_{\alpha\beta = 1}^n (\mathsf C_{\alpha}^2 + \mathsf C_{\beta}^2)\langle \mathsf B (\mathsf e_{\alpha} , \mathsf e_{\beta} ), \mathsf B (\mathsf e_{\alpha}  , \mathsf e_{\beta} ) \rangle \\
& =  \sum_{\alpha\beta = 1}^n 2 \mathsf C_{\alpha}^2 \langle \mathsf B (\mathsf e_{\alpha}  , \mathsf e_{\beta} ), \mathsf B (\mathsf e_{\alpha}  , \mathsf e_{\beta} ) \rangle.
 \end{aligned}
\end{equation}

Substituting \eqref{4.7}, \eqref{4.8} and \eqref{4.9} into \eqref{3.8}, we obtain the desired \eqref{4.1}.
\end{proof}

\begin{lem}[An Estimate on the Average Variation Formula \eqref{3.16} on the Domain of $u$]  Let $C_i$ be as in \eqref{4.12}. Then
\begin{equation}\label{4.10}
\begin{aligned}
\sum_{\ell=1}^q\frac{d^2}{dt^2}E_{\Phi}( u \circ f_t^{v_{\ell}^\top})_{\big |_{t=0}}& \le \int_M \sum_{i =1}^m  C_{i}^2 \sum_{j =1}^m \bigg (4 \langle   B (e_{i},  e_{j} ),  B (e_{i}  , e_{j} ) \rangle - \langle  B (e_{i}  , e_{i} ),  B(e_{j}  , e_{j} \rangle \bigg )\, dx\\
\end{aligned}
\end{equation}

\end{lem}
\begin{proof}[Proof of Lemma 4.2]
In the following, we consider  the product matrix $(u_{i \alpha}) \cdot (u_{i \alpha})^{\text T} $ to be an $m\times m$ symmetric  matrix with the $i, j$ entry over the field of real numbers given by
\begin{eqnarray}\label{4.11}
\left( \sum_{\alpha =1}^n u_{i\alpha}u_{j\alpha}\right).
\end{eqnarray}
Without ambiguity, we use the same notations for local orthonormal bases $\{e_1,\cdots,e_m\}$ in $M$ and $\{\mathsf e_1,\cdots,\mathsf e_n\}$ in $N$ so that the product matrix is diagonalizable. That is,
\begin{equation}\label{4.12}
\sum_{\alpha=1}^n u_{i \alpha}u_{j \alpha}=
\begin{cases}
\begin{aligned}
0\quad &\text{if}\quad  j \neq i\\
 C_{i}  \quad &\text{if}\quad  j = i
\end{aligned}
\end{cases}
\end{equation}
for some $ C_{i} \ge 0\, .$
Then by \eqref{4.2} and \eqref{1.1}, we have

\begin{equation}\label{4.200}
2 e(u) = \sum_{i=1}^m\langle du(e_i),du(e_i)\rangle =\sum_{i=1}^m C_{i},
\end{equation}

\begin{equation}\label{4.13}
\begin{aligned}
4 e_{\Phi}(u)&=\sum_{ij=1}^m\langle du(e_i),du(e_j)\rangle^2 =\sum_{ij=1}^m\sum_{\alpha\beta\gamma\delta=1}^n\langle u_{i\alpha} \mathsf e_{\alpha} ,u_{j\beta} \mathsf e_{\beta}\rangle \langle u_{i\gamma} \mathsf e_{\gamma} ,u_{j\delta} \mathsf e_{\delta}\rangle  \\
& =\sum_{ij=1}^m\sum_{\alpha\gamma=1}^n u_{i\alpha}u_{j\alpha}u_{i\gamma}u_{j\gamma} =\sum_{ij=1}^m (\sum_{\alpha=1}^n u_{i\alpha}u_{j\alpha}\sum_{\gamma=1}^nu_{i\gamma}u_{j\gamma}) =\sum_{i=1}^m  C_{i}^2
\end{aligned}
\end{equation}
and via \eqref{3.3},
\begin{equation}\label{4.14}
\begin{aligned}
du \big (A^{e_{\nu}} (e_i)\big )&=\sum_{k = 1}^mdu \big ( \langle A^{e_{\nu}} (e_{i}  ), e_{k} \rangle e_{k}\big ) \\
& = \sum_{k = 1}^m \sum_{\alpha = 1}^n u_{k\alpha} \langle  B(e_{i}  , e_{k} ), e_{\nu} \rangle \mathsf e_{\alpha}.\end{aligned}
\end{equation}

In view of \eqref{4.14}, \eqref{3.3} and \eqref{4.12}, we have

\begin{equation}\label{4.15}
\begin{aligned}
&\sum_{ij=1}^m\sum_{\nu= n+1}^q \langle du \big (2 A^{e_{\nu}} A^{e_{\nu}} (e_i)\big ), du(e_j) \rangle \langle du(e_i), du(e_j) \rangle\\
&=\sum_{ijk=1}^m\sum_{\nu= n+1}^q\sum_{\alpha\beta\gamma = 1}^n 2  \langle  B (e_{i}  , e_{k} ), e_{\nu} \rangle \langle du \big (A^{e_{\nu}} (e_{k})\big ), u_{j\alpha}\mathsf e_{\alpha}\rangle \langle u_{i\beta}\mathsf e_{\beta }, u_{j\gamma}\mathsf e_{\gamma}\rangle\\
& = \sum_{ijk\hbar = 1}^m 2 \bigg (\sum_{\beta=1}^n u_{i\beta} u_{j\beta}\sum_{\alpha=1}^nu_{\hbar\alpha}u_{j\alpha}\sum_{\nu= n+1}^q \langle  B (e_{i}  , e_{k} ), e_{\nu} \rangle  \langle  B (e_{k}  , e_{\hbar}), e_{\nu} \rangle \bigg ) \\
& = \sum_{ij = 1}^m 2  C_{i}^2 \langle  B (e_{i}  , e_{j} ), B (e_{i}  , e_{j}) \rangle.
 \end{aligned}
\end{equation}

Similarly, by \eqref{4.14}, \eqref{3.3} and \eqref{4.12}, we have

\begin{equation}\label{4.16}
\begin{aligned}
& \sum_{ij=1}^m\sum_{\nu= n+1}^q -\langle \text{tr} (A^{e_{\nu}} ) du \big ( A^{e_{\nu}} (e_i)\big ), du(e_j) \rangle \langle du(e_i), du(e_j) \rangle\\
&= \sum_{ijk\hbar=1}^m\sum_{\nu= n+1}^q\sum_{\alpha\beta\gamma\delta = 1}^n -\langle A^{e_{\nu}}(e_{\hbar} ), e_{\hbar} \rangle u_{k\alpha} \langle  B (e_{i}  , e_{k} ), e_{\nu} \rangle \langle \mathsf e_{\alpha}, u_{j\beta}\mathsf e_{\beta}\rangle \langle u_{i\gamma}\mathsf e_{\gamma }, u_{j\delta}\mathsf e_{\delta}\rangle\\
& = \sum_{ijk\hbar = 1}^m -\big (\sum_{\alpha=1}^n u_{k\alpha} u_{j\alpha}\sum_{\gamma=1}^nu_{i\gamma}u_{j\gamma}\sum_{\nu= n+1}^q \langle  B (e_{\hbar}  , e_{\hbar} ), e_{\nu} \rangle  \langle  B (e_{i}  , e_{k} ), e_{\nu} \rangle \big ) \\
& = \sum_{i\hbar = 1}^m - C_{i}^2 \langle  B (e_{\hbar}  , e_{\hbar} ), B (e_{i}  , e_{i}) \rangle \\
& = \sum_{ij = 1}^m - C_{i}^2 \langle  B (e_{i}  , e_{i} ), B (e_{j}  , e_{j}) \rangle.
 \end{aligned}
\end{equation}

In view of \eqref{4.14}, \eqref{3.3} and \eqref{4.12}, we have

\begin{equation}\label{4.17}
\begin{aligned}
&\sum_{ij=1}^m\sum_{\nu= n+1}^q 2 \langle  du \big ( A^{e_{\nu}} (e_i)\big ), du(e_j) \rangle \langle du(e_i), du \big ( A^{e_{\nu}} (e_j)\big ) \rangle
\end{aligned}
\end{equation}
\begin{equation*}
\begin{aligned}
&=  \sum_{ijk\hbar=1}^m\sum_{\nu= n+1}^q\sum_{\alpha\beta\gamma\delta = 1}^n  2 u_{k\alpha} \langle  B (e_{i}  , e_{k} ), e_{\nu} \rangle \langle
\mathsf e_{\alpha}, u_{j\gamma}\mathsf e_{\gamma}\rangle u_{\hbar\beta} \langle  B (e_{j}  , e_{\hbar} ), e_{\nu} \rangle
\langle u_{i\delta}\mathsf e_{\delta}, \mathsf e_{\beta}\rangle \\
& = \sum_{ij\hbar k = 1}^n 2 \big (\sum_{\alpha=1}^n u_{k\alpha} u_{j\alpha}\sum_{\beta=1}^nu_{\hbar\beta}u_{i\beta}\sum_{\nu= n+1}^q \langle  B (e_{i}  , e_{k} ),  B (e_{j}  , e_{\hbar} )
 \rangle \big ) \\
& = \sum_{ij = 1}^m 2  C_{i}  C_{j}\langle  B (e_i  , e_{j} ), B (e_{i}  , e_{j} ) \rangle \le \sum_{ij = 1}^m ( C_{i}^2 +  C_{j}^2)\langle B (e_{i} , e_{j} ),  B (e_{i}  , e_{j} ) \rangle \\
& =  \sum_{ij = 1}^m 2  C_{i}^2 \langle  B (e_{i},  e_{j} ),  B (e_{i}  , e_{j} ) \rangle.
 \end{aligned}
\end{equation*}

Substituting \eqref{4.15}, \eqref{4.16} and \eqref{4.17} into \eqref{3.16}, we obtain the desired
\eqref{4.10}.

\end{proof}
\begin{rem} In examining the factors in the estimates in \eqref{4.1} and \eqref{4.10},
\begin{equation}\label{4.20}
\begin{aligned}
& \sum_{\beta =1}^n 4 \langle  \mathsf B (\mathsf e_{\alpha} , \mathsf e_{\beta} ), \mathsf B (\mathsf e_{\alpha}  , \mathsf e_{\beta} ) \rangle - \langle \mathsf B (\mathsf e_{\alpha}  , \mathsf e_{\alpha} ), \mathsf B (\mathsf e_{\beta}  , \mathsf e_{\beta}) \rangle   < 0
\quad \operatorname{on}\quad N\\
\text{and}\quad & \sum_{i =1}^m 4 \langle  B (e_{i},  e_{j} ), B (e_{i}  , e_{j} ) \rangle - \langle B (e_{i}  , e_{i} ), B (e_{j}  , e_{j} )\rangle  < 0\quad \operatorname{on}\quad M,
\end{aligned}
\end{equation}
which yield information on $u$ into $N$ and from $M$ respectively, we find $\Phi$-superstrongly unstable manifold as defined in Definition \ref {D:1.3} or \eqref{1.5}.
\end{rem}

Now we are ready to apply an {\it algebraic} average method to prove \smallskip

\noindent
{\bf Theorem \ref{T:1.1} (a)}
{\it For every compact $\Phi$-$\operatorname{SSU}$ manifold $M$, there are no nonconstant smooth stable $\Phi$-harmonic map $u: M \to N\, $ into any compact manifold $N$}
\smallskip

\noindent
{\bf Theorem \ref{T:1.1} (c)}
{\it For every compact $\Phi$-$\operatorname{SSU}$ manifold $N$, there are no nonconstant smooth stable $\Phi$-harmonic map $u: M \to N\, $ from any compact manifold $M$ $ (u$ is not necessarily $\Phi$-harmonic$)$.}

\begin{proof}[Proof of Theorem \ref{T:1.1} $(a)$ and Theorem \ref{T:1.1} $(c)$]
If $u$ is not constant,
then by \eqref{4.100} $(\text{resp}. \eqref{4.200})$, there exists $1 \le \alpha \le n$ $(\text{resp}.\,  1 \le i \le m \big )$ and a domain $\mathsf D \subset N$ $(\text{resp}.\,  D \subset M),$ over which  $\mathsf C_{\alpha} > 0\, (\text{resp}.\,  C_i > 0 )\, .$ Hence, if $N$ is $\Phi$-SSU $(\text{resp}.\, M$ is $\Phi$-SSU$)$  \begin{equation}\label{4.21}
\begin{aligned}
& \mathsf C_{\alpha}^2 \sum_{\beta =1}^n \big (4 \langle  \mathsf B (\mathsf e_{\alpha},  \mathsf e_{\beta} ), \mathsf  B (\mathsf e_{\alpha}  , \mathsf e_{\beta} ) \rangle - \langle \mathsf  B (\mathsf e_{\alpha}  , \mathsf e_{\alpha} ), \mathsf  B (\mathsf e_{\beta}  , \mathsf e_{\beta} \rangle \big )< 0\quad \text{on}\quad \mathsf D\\
\bigg ( \text{resp}.\quad  & C_{i}^2 \sum_{j =1}^m \big (4 \langle  B (e_{i},  e_{j} ), B (e_{i}  , e_{j} ) \rangle - \langle B (e_{i}  , e_{i} ), B (e_{j}  , e_{j} ) \rangle \big ) < 0 \quad \text{on}\quad D \bigg )
\end{aligned}
\end{equation}
By Lemma \ref{L:4.1} (resp. Lemma \ref{4.2}) we have via \eqref{4.21} \begin{equation}\label{4.22}
\begin{aligned}
\sum_{\ell=1}^q\frac{d^2}{dt^2}E_{\Phi}( \mathsf f_t^{\mathsf v_{\ell}^\top} \circ u)_{\big |_{t=0}} & < 0\\
\big (\text{resp}. \sum_{\ell=1}^q\frac{d^2}{dt^2}E_{\Phi}( u \circ f_t^{v_{\ell}^\top})_{\big |_{t=0}} & < 0 \big ).\end{aligned}
\end{equation}
 By the algebraic average method, \eqref{4.22} implies that there exists a number $\ell\, ,$ $1 \le \ell \le q$ such that 
 $$\begin{aligned}
 & \frac{d^2}{dt^2}E_{\Phi}( \mathsf f_t^{\mathsf v_{\ell}^\top} \circ u)_{\big |_{t=0}} < 0.\, \text{Or}\, \sum_{\ell=1}^q\frac{d^2}{dt^2}E_{\Phi}( \mathsf f_t^{\mathsf v_{\ell}^\top} \circ u)_{\big |_{t=0}} \ge 0,\, \text {a}\, \text{contradiction}.\\
\bigg ( \text{resp}.\quad  & \frac{d^2}{dt^2}E_{\Phi}( u \circ f_t^{v_{\ell}^\top})_{\big |_{t=0}} < 0.\,  \text{Or}\, \sum_{\ell=1}^q\frac{d^2}{dt^2}E_{\Phi}( u \circ f_t^{v_{\ell}^\top})_{\big |_{t=0}}\ge 0,\, \text {a}\, \text{contradiction}.\bigg )
\end{aligned}
$$
This means that there exists a vector field $\mathsf v_{\ell}^{\top} = \mathsf e_{\ell}$ (resp. $v_{\ell}^{\top} =  e_{\ell}$) for some $1 \le \ell \le q$, along which the variation $$\begin{aligned}
& \mathsf f_t^{\mathsf v_{\ell}^\top} \circ u\, \text{decreases}\,  \text{the}\, \Phi\text{-energy}\, \text{of}\, \text{any}\,  \text{nonconstant}\,  \text{map}\, u\\
\bigg ( \text{resp}.\quad  & u \circ f_t^{v_{\ell}^\top} \text{decreases}\,  \text{the}\, \Phi\text{-energy}\, \text{of}\, \text{any}\,  \text{nonconstant}\,  \text{map}\, u.
\bigg )
\end{aligned}$$ That is, 
$$\begin{aligned}
u,\quad \text{not}\, \text{necessarily}\, & \Phi\text{-harmonic}\quad \text{is}\, \, \text{not}\, \, \text{a}\, \, \text{nonconstant}\, \, \Phi\text{-stable}\, \, \text{map}\\
\bigg ( \text{resp}.\quad  & u\quad \text{is}\, \, \text{not}\, \, \text{a}\, \, \text{nonconstant}\, \, \Phi\text{-stable}\, \, \text{map}
\bigg ) .
\end{aligned}$$
\end{proof}

\noindent
{\bf Theorem \ref{T:1.1} (b)}
{\it If $N$ is $\Phi$-$\operatorname{SSU}$, then for every compact manifold $M$, the homotopic class of any map from $M$ into $N$ contains elements of arbitrarily small $\Phi$-energy.}
\smallskip

\begin{lem}\label{L:4.3}
If $N$ is a compact $\Phi$-SSU manifold, then there is a number $0<\rho<1$ such that for any compact manifold $M$ and any map $u:M\rightarrow N$ there is a map $u_1:M\rightarrow N$ homotopic to $u$ with $\Phi(u_1)\leq \rho \Phi(u)$.
\end{lem}
\begin{proof}
Let $\widehat{T}_yN$ be the space of the unit tangent vectors to $N$ at the point $y \in N$. Since $N$ is $\Phi$-SSU and $\widehat{T}_yN$ is compact, by \eqref{1.5}, there exists $\kappa > 0$ such that for every $y \in N\, $ and every $\mathsf x \in \widehat{T}_yN\, ,$
\begin{align}\label{4.23}
\mathsf F_y(\mathsf x) < - q \kappa
\end{align}

It follows from \eqref{4.1}, \eqref{1.5}, \eqref{4.5} and \eqref{4.23} that
\begin{align}
\sum_{\ell=1}^q\frac{d^2}{dt^2}\Phi(\mathsf f_t^{\mathsf v_{\ell}^\top}\circ u)_{\big |_{t=0}}\le \int_M \sum_{\alpha =1}^n \mathsf C_{\alpha}^2 \mathsf F_{u(x)}(\mathsf e_{\alpha})\, dx \le  -q\kappa\int_M e_{\Phi} (u)\, dx=-4q\kappa\Phi(u).\label{4.24}
\end{align}
We now proceed in steps.\\

{\bf{Step 1}}. There is a number $\xi\geq 4\kappa>0$ such that for $1\leq \ell\leq q $, $|t|\leq 1$ and all $\mathsf X,\mathsf Y\in \Gamma(TN)$,
\begin{align}
\bigg | \frac{d^3}{dt^3} \langle d\mathsf f_t^{\mathsf v_{\ell}^\top}(\mathsf X),d\mathsf f_t^{\mathsf v_{\ell}^\top}(\mathsf Y)\rangle^2\bigg |\leq \xi |\mathsf X|^2|\mathsf Y|^2.\label{4.25}
\end{align}
\begin{proof}
Let $SN$ be the unit sphere bundle of $N$. Then the function defined on the compact set $[-1,1]\times SN\times SN$ by
\begin{align}\label{4.26}
(t,\mathsf x,\mathsf y)\mapsto \max_{1\leq \ell\leq q}\left|\frac{d^3}{dt^3}\langle d\mathsf f_t^{\mathsf v_{\ell}^\top}(\mathsf x),d\mathsf f_t^{\mathsf v_{\ell}^\top}(\mathsf y)\rangle^2\right|
\end{align}
is continuous and thus has a maximum. Let ${\xi}_0$ be this maximum and $\xi=\max\{\frac{3\kappa}{m}, \xi_0\}$. Then \eqref{4.25} follows by homogeneity.
\end{proof}
{\bf{Step 2}}. There is a smooth vector field $\mathsf V$ on $N$ such that if $\xi$ is as in Step 1,
then we have
\begin{align}
\frac{d}{dt}\Phi(\mathsf f_t^{\mathsf V}\circ u)_{\big |_{t=0}}\leq 0,\label{4.27}
\end{align}
\begin{align}
\frac{d^2}{dt^2}\Phi(\mathsf f_t^{\mathsf V}\circ u)_{\big |_{t=0}}\leq - 4\kappa\Phi(u),\label{4.28}
\end{align}
and
\begin{align}
\left|\frac{d^3}{dt^3}\Phi(\mathsf f_t^{\mathsf V}\circ u)\right|\leq m\xi\Phi(u),\quad{\text{for}}\quad |t|\leq 1.\label{4.29}
\end{align}
\begin{proof}
From \eqref{4.24} it is seen that
\begin{align}\label{4.30}
\frac{d^2}{dt^2}\Phi(\mathsf f_t^{\mathsf v_{\ell}^\top}\circ u)_{\big |_{t=0}}\leq - 4\kappa\Phi(u),
\end{align}
for some $1 \le \ell \le q$. Or we would have $\sum_{\ell=1}^q\frac{d^2}{dt^2}\Phi(\mathsf f_t^{\mathsf v_{\ell}^\top}\circ u)_{\big |_{t=0}}  > -4q\kappa\Phi(u)\, ,$ contradicting \eqref{4.24}. If $\frac{d}{dt}\Phi(\mathsf f_t^{\mathsf v_{\ell}^\top}\circ u)_{\big |_{t=0}}\leq 0,$ set $\mathsf V=\mathsf v_{\ell}^\top$; otherwise, set $\mathsf V=-\mathsf v_{\ell}^\top$. Then \eqref{4.27} and \eqref{4.28} hold. From \eqref{4.25}, we have
\begin{equation}\label{4.31}
\begin{aligned}
\left|\frac{d^3}{dt^3}\Phi(\mathsf f_t^{\mathsf V}\circ u)\right| & = \frac 14 \int_M\sum_{i,j=1}^m \frac{d^3}{dt^3} \langle d\mathsf f_t^{\mathsf v_{\ell}^\top}\big (du(e_i)\big ),d\mathsf f_t^{\mathsf v_{\ell}^\top}\big (du(e_j)\big )\rangle^2\, dx\\
& \leq \frac {\xi}{4} \int_M\sum_{i,j=1}^m|du(e_i)|^2|du(e_j)|^2\, dx\\
& =\frac {\xi}{4}\int_M\bigg ( \sum_{i=1}^m\langle du(e_i), du(e_i)\rangle\bigg )^2\, dx\\ & \le \frac {m\xi}{4}\int_M \sum_{i=1}^m\langle du(e_i), du(e_i)\rangle^2\, dx\\
&\le \frac {m\xi}{4}\int_M \sum_{i,j=1}^m\langle du(e_i), du(e_j)\rangle^2\, dx = m\xi\Phi(u).
\end{aligned}
\end{equation}
So \eqref{4.29} is right.
\end{proof}
{\bf{Step 3}}. Let $\zeta=\frac{3\kappa}{m\xi}$ ($\zeta\leq 1$, as $\frac{3\kappa}{m}\leq \xi$), $\rho=1-\frac{\kappa \zeta^2}{2}$, and $\mathsf V$ be as in Step 2. Then $0<\rho<1$ and
\begin{align}\label{4.32}
\Phi(\mathsf f_{\zeta}^{\mathsf V}\circ u)\leq \rho \Phi(u).
\end{align}
\begin{proof}
Let $\Phi(t)=\Phi(\mathsf f_t^\mathsf V\circ u)$. Then by Step 2 for $0\leq t\leq \zeta$, we have
\begin{align}\label{4.33}
&\Phi^{\prime \prime}(t)=\Phi^{\prime \prime}(0)+\int_0^t\Phi^{\prime \prime \prime}(s)ds\leq -4\kappa\Phi(u) + m \xi \zeta\Phi(u)=-\kappa\Phi(u).
\end{align}
Thus
\begin{align}\label{4.33}
\Phi^{\prime}(t)=\Phi^{\prime}(0)+\int_0^t\Phi^{\prime \prime}(s)ds\leq -\kappa t\Phi(u)
\end{align}
and
\begin{align}\label{4.34}
\Phi(\zeta)=\Phi(0)+\int_0^{\zeta}\Phi^{\prime}(s)ds\leq \left(1-\frac{\kappa\zeta^2}{2}\right)\Phi(u)=\rho\Phi(u).
\end{align}
As both $\Phi(\zeta)$ and $\Phi(u)$ are positive this inequality implies $\rho$ is positive. Let $u_1=\mathsf f_{\zeta}^{\mathsf V}\circ u$. Then $u_1$ is homotopic to $u$ and we have just shown $\Phi(u_1)\leq \rho\Phi(u)$.
\end{proof}
From Step 1, Step 2 and Step 3, we know this lemma is right.
\end{proof}

\begin{proof}[Proof of Theorem \ref{T:1.1} $($b$)$]
Let $u:M\rightarrow N$ be any smooth map from $M$ to $N$. By using Lemma \ref{L:4.3}, we can find a map $u_1:M\rightarrow N$ which is homotopic to $f$ with $\Phi(u_1)\leq \rho\Phi(u)$.
Another application of the lemma gives an $u_2$ homotopic to $u_1$ with $\Phi(u_2)\leq \rho\Phi(u_1)\leq \rho^2\Phi(u)$. By induction, there is $u_{\ell}$ $(\ell=1,2,\cdots)$ homotopic to $u$ with $\Phi(u_l)\leq \rho^l\Phi(u)$.
But $0<\rho<1$ whence $\lim_{\ell\rightarrow\infty}\Phi(u_{\ell})=0$ as required.
\end{proof}

\begin{cor}\label{C:4.1}
If $N$ is $\Phi$-$\operatorname{SSU}$. then the infimum of the $\Phi$-energy is zero among maps homotopic to the identity map on $N\, .$
\end{cor}
\begin{proof} This follows at once from Theorem 1.1$($b$)$ by choosing $M=N$ and the smooth map to be the identity map on $N\, .$
\end{proof}

\noindent
{\bf Theorem \ref{T:1.1} (d)}
{\it If $N$ is $\Phi$-$\operatorname{SSU}$, then for every compact manifold $M$, the homotopic class of any map from $N$ into $M$ contains elements of arbitrarily small $\Phi$-energy.}
\smallskip

\begin{lem}\label{L: 4.4}
If $N$ is a manifold such that the infimum of the energy is zero among maps homotopic to the identity and if $M$ is a compact manifold, then
 the infimum of the energy is zero in each homotopy class of maps from $N$ to $M$.
\end{lem}
\begin{proof}[Proof of Lemma \ref{L: 4.4}]
Let $K, M, N$ be three compact Riemannian manifolds of dimensions $k, m$ and $n$ respectively, and $K \xrightarrow {\psi} M \xrightarrow {u} N$ be smooth maps.
Denote an $m$$\times$$m$-matrix $U$ with the $i$-$j$ entry $U_{ij}$ given by $U_{ij} = \langle du(e_i), du(e_j)\rangle\, ,$ and an $k$$\times$$k$-matrix $\Psi$ with the $i$-$j$ entry $\Psi_{ij}$ given by $\Psi_{ij} = \langle(d\psi(e_i), d\psi(e_j)\rangle_M\, .$ Then the $\Phi$-energy density $e_{\Phi}(u)$ of $u$ satisfies $e_{\Phi}(u) =\frac 14 \sum_{i,j=1}^m\langle du(e_i),du(e_j)\rangle^2\, = \frac 14 \text {trace} (U \cdot U^T)\, .$ Similarly, $e_{\Phi}(\psi) =\frac 14 \text {trace} (\Psi \cdot \Psi^T)\, ,$ and
$$ \begin{aligned}e_{\Phi}(u \circ \psi ) &=\frac 14 \text {trace} \big ( (U \cdot \Psi) \cdot (U \cdot \Psi)^T \big )\\
&=\frac 14 \text {trace} \big ( U \cdot ( \Psi \cdot \Psi ^T) \cdot U^T \big )\\
&=\frac 14 \text {trace} \big ( (U \cdot U^T) \cdot ( \Psi \cdot \Psi ^T)  \big ).
\end{aligned}$$
Thus, if $u: M \to N$ be any smooth map, its composition with $\psi ^{\ell} : M \to M$ homotopic to the identity map on $M$ and $E_{\Phi} (\psi^{\ell}) \to 0\, ,$ as $\ell \to \infty$, then  $e_\Phi (\psi^{\ell}) \to 0\, ,$ as $\ell \to \infty$, since $E_{\Phi} (\psi^{\ell}) = \int _M e_\Phi (\psi^{\ell})\, dx\, .$  Hence, each nonnegative entry $(\Psi ^{\ell} _{ij}) ^2\to 0\, ,$ where $\Psi^{\ell}_{ij} = \langle d\psi^{\ell}(e_i), d\psi^{\ell}(e_j)\rangle_M \, .$ For $\sum _{i,j=1}^m  (\Psi ^{\ell} _{ij} )^2 \to 0\, .$ Or  $\sum _{i,j=1}^m  (\Psi ^{\ell} _{ij} )^2 \not \to 0\, ,$ contradicting the hypothesis $E_{\Phi} (\psi^{\ell}) \to 0\, ,$ as $\ell \to \infty\, .$
It follows that each entry  $\Psi ^{\ell} _{ij} \to 0\, .$ Furthermore, the sequence $u \circ \psi ^{\ell}$ is homotopic to $u$ with
\[ \begin{aligned}\Phi(u \circ \psi ^{\ell})&  = \frac 14 \int _M \text {trace} \big ( (U \cdot U^T) \cdot ( \Psi ^{\ell} \cdot (\Psi ^{\ell} )^T)  \big )\, dv \\
& = \frac 14 \int _M \sum_{s,i=1}^m \big ( (U \cdot U^T)_{si}  ( \Psi ^{\ell} \cdot (\Psi ^{\ell} )^T)_{is}  \big )\, dv \\
& = \frac 14 \int _M \sum_{s,i=1}^m \bigg ( \big ( \sum _{t=1}^n U_{st}  U_{it}) ( \sum _{j=1}^m \Psi ^{\ell} _{ij} \Psi ^{\ell} _{sj}  \big ) \bigg )\, dv \\
& \to 0\quad \text{as}\quad  \ell \to \infty\, ,
\end{aligned}\]
$\text{since}\quad \text{each}\quad \Psi ^{\ell} _{ij} \to 0\quad \text{and}\quad  \text{each}\quad U_{st}\quad {is}\quad {bounded}\quad {on}\quad M.$\end{proof}

\begin{proof}[Proof of Theorem 1.1 $(d)$] This follows at once from Corollary \ref{C:4.1} and Lemma \ref{L: 4.4}.
\end{proof}

\section{Examples of $\Phi$-SSU manifolds }
\begin{thm}\label{T:5.1}
Let $N$ be a hypersurface in Euclidean space. Then $N$ is $\Phi$-SSU  if and only if its principal curvatures satisfy $$0<\lambda_1\leq \lambda_2\leq \cdots\leq \lambda_n < \frac{1}{3}(\lambda_1+\cdots+\lambda_{n-1})$$.
\end{thm}
\begin{proof}
This follows by using the same method in proving theorem 3.3 in \cite{WY} or simplifying  \e{1.5}.
\end{proof}
\begin{cor}\label{C:5.1}
The graph of $f(x)=x_1^2+\cdots+x_n^2\, ,  x = (x_1, \dots , x_n) \in \mathbb R^n$ is $\Phi$-SSU if and only if $n>4$.
\end{cor}

\begin{cor}\label{C:5.2}
The standard sphere $S^n$ is $\Phi$-SSU if and only if $n>4$.
\end{cor}

\begin{cor}\label{C:5.3}  Let $N$ be a hypersurface in Euclidean space. Then

\noindent
$(i)$ If $N$ is $\Phi$-SSU, then $N$  is $p$-SSU for $2 \le p \le 4$.\\
$(ii)$ If $N$ is $p$-SSU for $p \ge 4$, then $N$  is $\Phi$-SSU.\\
$(iii)$ $N$ is $\Phi$-SSU if and only if $N$  is $4$-SSU.\\
\end{cor}
\begin{proof} $(i)$ and $(ii)$ follow from Theorem \ref{T:5.1}, and Theorem 3.3 in \cite {WY}, which states that a Euclidean hypersurface of dimension $n$ is $p$-SSU if and only if its principal curvatures satisfy $0<\lambda_1\leq \lambda_2\leq \cdots\leq \lambda_n < \frac{1}{p-1}(\lambda_1+\cdots+\lambda_{n-1})$. Combining $(i)$ and $(ii)$, we prove $(iii)$. 
\end{proof}
\begin{thm}
Let $\widetilde{N}$ be a compact convex hypersurface of $\mathbb R^q$ and $N$ be a compact connected minimal $k$-submanifold of $\widetilde{N}$. Assuming that the principal curvatures $\lambda_i$ of $\widetilde{N}$ satisfy $0<\lambda_1\leq\cdots\leq\lambda_{q-1}$. If $\<\text{Ric}^N(\mathsf x),\mathsf x\> > \frac{3}{4}k\lambda^2_{q-1}$ for any unit tangent vector $\mathsf x$ to $N$, then $N$ is $\Phi$-SSU.
\end{thm}
\begin{proof}
By assumption, $N$ is a submanifold of $\mathbb R^q$. From Gauss equation, we have
\begin{align}\label{5.2}
\mathsf B(\mathsf X,\mathsf Y)=\mathsf B_1(\mathsf X,\mathsf Y)+\widetilde{\mathsf B}(\mathsf X,\mathsf Y)\nu,
\end{align}
where $\mathsf B, \mathsf B_1,$ and $\widetilde{\mathsf B}$ denote the second fundamental form of $N$ in $\mathbb R^{q}$, $N$ in $\widetilde{N}$, and $\widetilde{N}$ in $\mathbb R^{q}$ respectively, and $\nu$ denotes the unit normal field of $\widetilde{N}$ in $\mathbb R^{q}$. Since $N$ is a minimal submanifold of $\widetilde{N}$, we have 
\begin{align}\label{5.3}
\sum_{i=1}^k\mathsf B(\mathsf e_i,\mathsf e_i)=\sum_{i=1}^k\mathsf B_1(\mathsf e_i,\mathsf e_i)
+\sum_{i=1}^k\widetilde{\mathsf B}(\mathsf e_i,\mathsf e_i)\nu=\sum_{i=1}^k\widetilde{\mathsf B}(\mathsf e_i,\mathsf e_i)\nu,
\end{align}
where $\{\mathsf e_i\}_{i=1}^{k}$ is a local orthonormal frame on $N$ such that $\widetilde{\mathsf B}(\mathsf e_i,\mathsf e_j)=\lambda_i\delta_{ij}$.

It follows from Gauss equation, \eqref{5.2}, \eqref{5.3},  and curvature assumption that for any unit tangent vector $\mathsf x$ to $N$,
\begin{equation}
\begin{aligned}
\mathsf F_y ({\mathsf x}) &=\sum_{i=1}^k\big (4\langle \mathsf  B(\mathsf x,\mathsf e_i),\mathsf  B(\mathsf x,\mathsf e_i)\rangle-\langle \mathsf  B(\mathsf x,\mathsf x),\mathsf  B(\mathsf e_i,\mathsf e_i)\rangle\big )\nonumber\\
&=-4\langle \operatorname{Ric}^N(\mathsf x),\mathsf x\rangle + 3\sum_{i=1}^n\langle \mathsf  B(\mathsf x,\mathsf x),\mathsf B(\mathsf e_i,\mathsf e_i)\rangle\nonumber\\
&=-4\langle \operatorname{Ric}^N(\mathsf x),\mathsf x\rangle+3\sum_{i=1}^k\widetilde{\mathsf  B}(\mathsf x,\mathsf x)\widetilde{\mathsf  B}(\mathsf e_i,\mathsf e_i)\nonumber\\
&=-4\langle \operatorname{Ric}^N(\mathsf x),\mathsf x\rangle+3(\lambda_1+\cdots+\lambda_k)\widetilde{\mathsf B}(\mathsf x,\mathsf x)\nonumber\\
&\leq-4\langle \operatorname{Ric}^N(\mathsf x),\mathsf x\rangle+3(\lambda_1+\cdots+\lambda_k)\lambda_{q-1}\nonumber\\
&\leq-4\langle \operatorname{Ric}^N(\mathsf x),\mathsf x\rangle+3k\lambda^2_{q-1} < 0,\nonumber
\end{aligned}
\end{equation}
Consequently, $N$ is $\Phi$-SSU.
\end{proof}

\begin{lem}(\cite{11,7})\label{L:5.1}
Let $E^{q-1}$ be the ellipsoid given by 
\begin{equation}
E^{q-1}=\left\{(x_1,\cdots,x_{q})\in \mathbb R^{q}:\frac{x_1^2}{a_1^2}+\cdots+\frac{x_{q}^2}{a_{q}^2}=1, a_i>0, 1\leq i\leq q\right\}.
\end{equation}
Suppose $\{\lambda_i\}_{i=1}^{q-1}$ is a family of principal curvatures of $E^{q-1}$ in $\mathbb R^{q}$ satisfying $0<\lambda_1\leq \cdots \le \lambda_{q-1}$. Then
we have
\begin{eqnarray}
\quad \frac{\min_{1\leq i\leq q}\{a_i\}}{(\max_{1\leq i\leq q}\{a_i\})^2}\leq \lambda_1\leq \cdots\leq \lambda_{q-1}\leq \frac{\max_{1\leq i\leq q}\{a_i\}}{(\min_{1\leq i\leq q}\{a_i\})^2}.
\end{eqnarray}
\end{lem}

\begin{thm}\label{T:5.2}
A compact minimal  $k$-submanifold $N$ of an ellipsoid $E^{q-1}$ in $\mathbb R^{q}$ with $\langle \operatorname{Ric}^N(\mathsf x),\mathsf x\rangle > \frac{3}{4}\frac{(\max_{1\leq i\leq q}\{a_i\})^2}{(\min_{1\leq i\leq q}\{a_i\})^4}k$ for any unit tangent vector $\mathsf x$ to $N$ is $\Phi$-SSU.
\end{thm}
\begin{proof} This follows at once from Theorem \ref{T:5.2} and Lemma \ref{L:5.1}.
\end{proof}

\begin{cor}\label{C:5.4}
A compact minimal $k$-submanifold $N$ of the unit sphere $S^{q-1}$ with $\langle \operatorname{Ric}^N(\mathsf x),\mathsf x\rangle > \frac{3}{4}k$ for any unit tangent vector $\mathsf x$ to $N$ is $\Phi$-SSU.
Thus $N$ satisfies Theorem \ref{T:1.1} $(a)$, $(b)$, $(c)$, and $(d)$.\end{cor}
\begin{proof} This follows at once from Theorem \ref{T:5.2} in which $a_1 = \cdots = a_q =1\, .$
\end{proof}
\begin{remark} The case $N$ as in Corollary \ref{C:5.4} satisfies Theorem \ref{T:1.1} $(a)$ and $(c)$ are due to S. Kawai and N. Nakauchi (cf. \cite{KN2}).\label{R:5.1}\end{remark}
\begin{thm}
Let $N$ be a compact $k$-submanifold of the unit sphere $S^{q-1}$ in $\mathbb R^{q}$, $k > 4$, and Let $\mathsf B_1$ be the second fundamental form of $N$ in $S^{q-1}$. If
\begin{align}
||\mathsf B_1||^2<\frac{k-4}{\sqrt{k}+4},\nonumber
\end{align}
then $N$ is $\Phi$-SSU.
\end{thm}
\begin{proof}
From \eqref{5.2} (in which $\widetilde{N} = S^{q-1}$) and Cauchy-Schwarz inequality, for any unit tangent vector $\mathsf x$ to $N$, we have
\begin{align}
\mathsf F_y ({\mathsf x}) &=\sum_{i=1}^k\big (4\langle \mathsf B(\mathsf x,\mathsf e_i),\mathsf B(\mathsf x,\mathsf e_i)\rangle-\langle \mathsf B(\mathsf x,\mathsf x),\mathsf B(\mathsf e_i,\mathsf e_i)\rangle\big )\nonumber\\
&=\sum_{i=1}^k\big (4\langle \mathsf B_1(\mathsf x,\mathsf e_i),\mathsf B_1(\mathsf x,\mathsf e_i)\rangle-\langle \mathsf B_1(\mathsf x,\mathsf x),\mathsf B_1(\mathsf e_i,\mathsf e_i)\rangle\big )-(k-4)\nonumber\\
&\le 4 \sum_{i=1}^k \langle \mathsf B_1(\mathsf x,\mathsf e_i),\mathsf B_1(\mathsf x,\mathsf e_i)\rangle + | \mathsf B_1(\mathsf x,\mathsf x) | \big ( \big |\sum_{i=1}^k \mathsf B_1(\mathsf e_i,\mathsf e_i) \big |^2 \big )^{\frac 12} -(k-4)\nonumber\\
&\leq4||\mathsf B_1||^2+\sqrt{k}|\mathsf B_1(\mathsf x,\mathsf x)|\big (\sum_{i=1}^k\big |\mathsf B_1(\mathsf e_i,\mathsf e_i) \big |^2\big )^{\frac{1}{2}}-(k-4)\nonumber\\
&\leq(4+\sqrt{k})||\mathsf B_1||^2-(k-4)<0.\nonumber
\end{align}
So we obtain that $N$ is $\Phi$-SSU.
\end{proof}

\section{$\Phi$-SSU manifolds, $p$-SSU manifolds,  and Topology }

\begin{thm}\label{T: 6.1}
$(i)$ Every $\Phi$-SSU manifold $N$ is $p$-SSU for any $2 \le p \le 4$ \\
$(ii)$  Every compact $\Phi$-SSU manifold $N$ is $4$-connected, i.e. $\pi_1(N) = \cdots = \pi_{4}(N) = 0$.
\end{thm}
\begin{proof}
$(i)$ From \eqref{1.5}, we have
\begin{align}
\mathsf F_y(\mathsf x)=\sum_{i=1}^n\big (4\langle \mathsf B(\mathsf x,\mathsf e_i),\mathsf B(\mathsf x,\mathsf e_i)\rangle-\langle \mathsf B(\mathsf x,\mathsf x),\mathsf B(\mathsf e_i,\mathsf e_i)\rangle\big )<0\nonumber\\
\end{align}
for all unit tanget vector $\mathsf x \in T_y(N)$. It follows that
\begin{align}
 \mathsf F_{p,y}(\mathsf x)&=(p-2)\langle \mathsf B(\mathsf x,\mathsf x),\mathsf B(\mathsf x,\mathsf x)\rangle+\sum_{i=1}^n\big (2\langle \mathsf B(\mathsf ,\mathsf e_i),\mathsf B(\mathsf x,\mathsf e_i)\rangle-\langle \mathsf B(\mathsf x,\mathsf x),\mathsf B(\mathsf e_i,\mathsf e_i)\rangle\big ) \nonumber\\
&\leq\sum_{i=1}^n\big (p\langle \mathsf B(\mathsf x,\mathsf e_i),\mathsf B(\mathsf x,\mathsf e_i)\rangle-\langle \mathsf B(\mathsf x,\mathsf x),\mathsf B(\mathsf e_i,\mathsf e_i)\rangle\big )\nonumber \\
&\leq\sum_{i=1}^n\big (4\langle \mathsf B(\mathsf x,\mathsf e_i),\mathsf B(\mathsf x,\mathsf e_i)\rangle-\langle \mathsf B(\mathsf x,\mathsf x),\mathsf B(\mathsf e_i,\mathsf e_i)\rangle\big )<0\nonumber,
\end{align}
for $2\leq p\leq 4$.
So by \eqref{1.4}, $N$ is $p$-SSU for any $2\leq p\leq 4$.\\

$(ii)$ Since every compact $p$-SSU is $[p]$-connected (cf.  \cite [Theorem 3.10 , p. 645]{13}), $(ii)$ follows from $(i)$. 
\end{proof}
As an immediate application of Theorem \ref{T: 6.1}, we make the following
\begin{remark}
All examples of compact $\Phi$-SSU manifolds as discussed in Section 5, i.e., manifolds as in Theorems \ref{T:5.2}, 5.3, and 5.4 and Corollaries \ref{C:5.2}, and \ref{C:5.4}  are $4$-connected.
\end{remark}\label{R:6.1}

\section{The dimension of $\Phi$-SSU manifolds}

Now we recall a theorem in \cite{WY} states that
the dimension of any $p$-SSU manifold $N$ is greater than $p\, , p\ge 2$. Analogously, we have

\begin{thm}\label{T:7.1}
The dimension of any compact $\Phi$-SSU manifolds $N$ is greater than $4$.
\end{thm}
\begin{proof}[First Proof:]
Following the idea in \cite {WY} by using a maximum principle, we assume  
$|\mathsf B(\mathsf x,\mathsf x)|^2=\max_{\mathsf y\in \widehat{T}_yN}|\mathsf B(\mathsf y,\mathsf y)|^2$ 
for a fixed point $y\in N$, where $\widehat{T}_yN$ is the set of all unit tangent vectors to $N$ at the point $y$.
Let $C(t)$ be a curve in $\widehat {T}_yN$ such that $C(0)=\mathsf x$ and $C^{\prime}(0)=\mathsf y$.
Define $\psi(t)=|\mathsf B\big (C(t),C(t)\big )|^2$. Then $\psi^{\prime}(0)=0$ yields $\langle \mathsf B(\mathsf x,\mathsf x),\mathsf B(\mathsf x,\mathsf y)\rangle=0$ for all $\mathsf y\in \widehat{T}_yN$ such that $\langle \mathsf x,\mathsf y\rangle_N=0$.

Now let $\{\mathsf e_1=\mathsf x,\mathsf e_2,\cdots,\mathsf e_n\}$ be a local orthonormal frame of $N$ near $y$, and let $C^{\prime\prime}(0) = c_1 \mathsf x + c_2 \mathsf e_2 + \cdots + c_n \mathsf e_n\, .$ Then $c_1 = \langle C^{\prime\prime}(0), C(0)\rangle = - \langle C^{\prime}(0), C^{\prime}(0)\rangle = -1$ and $\langle \mathsf B\big (C^{\prime\prime}(0),\mathsf x\big ),\mathsf B(\mathsf x,\mathsf x)\rangle = \langle \mathsf B(c_1 \mathsf x + c_2 \mathsf e_2 + \cdots + c_n \mathsf e_n,\mathsf x),\mathsf B(\mathsf x,\mathsf x)\rangle = - \langle \mathsf B(\mathsf x,\mathsf x),\mathsf B(\mathsf x,\mathsf x)\rangle\, .$
Hence, by a maximum principle, $\psi^{\prime\prime}(0) \le 0$ which gives \begin{align}|\mathsf B(\mathsf x,\mathsf x)|^2\geq 2|\mathsf B(\mathsf x,\mathsf y)|^2+\langle \mathsf B(\mathsf x,\mathsf x),\mathsf B(\mathsf y,\mathsf y)\rangle \label{7.1}
\end{align} 
for all $\mathsf y\in\widehat{T}_yN$ such that $\langle \mathsf x,\mathsf y\rangle_N=0$. Substituting $\mathsf y = \mathsf e_i$ into \eqref{7.1} and summing over $i$ from $2$ to $n$, we have
\begin{align}
0 \geq - (n-1)|\mathsf B(\mathsf x,\mathsf x)|^2 + \sum_{i=2}^n\big (2 |\mathsf B(\mathsf x,\mathsf e_i)|^2 + \langle \mathsf B(\mathsf x,\mathsf x),\mathsf B(\mathsf e_i,\mathsf e_i)\rangle\big ).\label{7.2}
\end{align}
Since $N$ is $\Phi$-SSU, we have, via \eqref{1.5}
\begin{equation}
\begin{aligned}
0 & >  \sum_{i=1}^n\big (4\langle \mathsf B(\mathsf x,\mathsf e_i),\mathsf B(\mathsf x,\mathsf e_i)\rangle-\langle \mathsf  B(\mathsf x,\mathsf x),\mathsf B(\mathsf e_i,\mathsf e_i)\rangle\big )\\
& =  3 |\mathsf B(\mathsf x,\mathsf x)|^2 + \sum_{i=2}^n\big (4 |\mathsf B(\mathsf x,\mathsf e_i)|^2 - \langle \mathsf  B(\mathsf x,\mathsf x),\mathsf B(\mathsf e_i,\mathsf e_i)\rangle\big ).\end{aligned}\label{7.3}\end{equation}
Adding both sides of \eqref{7.2} and \eqref{7.3}, we have
\begin{align}
&&0 > (4-n)|\mathsf B(\mathsf x,\mathsf x)|^2 + 6\sum_{i=2}^n |\mathsf B(\mathsf x,\mathsf e_i)|^2 \label{7.4},
\end{align}
so we have $n > 4$.
\end{proof}

\begin{proof}[Second Proof:]
If $n \le 4\, ,$ then $N$ is not an $4$-SSU manifold, and hence not a $\Phi$-SSU manifold by Theorem \ref{T: 6.1}.
This completes the proof by showing that if the dimension of a compact $\Phi$-SSU manifold $N$ is less than or equal to $4$, then $N$ is not a $\Phi$-SSU manifold.
\end{proof}

\begin{remark}
Theorem \ref{T:7.1} is sharp, as $S^n$ is not $\Phi$-SSU for $n \le 4$ and is $\Phi$-SSU for $n > 4\, .$
\end{remark}
\begin{thm}[Sphere Theorem]\label{T:7.2}
Every compact $\Phi$-$\text{SSU}$ manifold $N$ of dimension $n < 10$ is homeomorphic to an $n$-sphere.
\end{thm}
\begin{proof}
In view of Theorem \ref{T: 6.1}, $N$ is 4-connected. By the Hurewicz isomorphism theorem, the 4-connectedness of $N$ implies homology groups $H_1(N)=\cdots=H_4(N)=0$. It follows from Proincare Duality Theorem and the Hurewicz Isomorphism Theorem (\cite{SP}) again, $H_{n-4}(N)=\cdots=H_{n-1}(N)=0$, $H_{n}(N)\neq 0\, , n < 10$ and $N$ is ($n-1$)-connected. Hence $N$ is a homotopy $n$-sphere, $n < 10$. Since $N$ is $\Phi$-SSU, $n\geq 5$ by Theorem 7.1. Consequently, a homotopy $n$-sphere $N$ for $n \ge 5$ is homeomorphic to an $n$-sphere by a Theorem of Smale (\cite{Sm}).
\end{proof}
\section{The identity map}

The identity map on every Riemannian manifold $M$ is $p$-harmonic for $1 < p < \infty\, $(cf. \cite [Proposition 4.1, p.652] {13}), and is $\Phi$-harmonic by Corollary \ref{C:2.3}.
Just as we study $p$-stability, $p$-index and $p$-nullity of the identity map, so do we explore $\Phi$-stability, $\Phi$-index and $\Phi$-nullity of the identity map. We discuss their parallelisms in this section, following the framework and ideas in \cite{13}.
 
The $E_{\Phi}$-Hessian of the identity map $\text {Id}$ on $\Gamma(\text {Id}^{-1}TM)$ is defined by
$$H^{E_{\Phi}}_{\text {Id}}(v,w)=\frac{\p^2E_{\Phi}(\text {Id}_{s,t})}{\p t \p s}_{\big |_{(s,t)=(0,0)}}.$$
Write the associated quadratic form
\begin{align} \vp_{_\Phi}(v)=H^{E_\Phi}_{\text {Id}}(v,v)\label{8.1}\end{align}
for short. It follows from Corollary \ref{C:2.4} that
\begin{equation}\label{8.2}
\begin{aligned}
& H^{E_{\Phi}}_{\text {Id}}(v,w)  =  \int _M \sum_{i,j=1}^m \langle \nabla  _{e_i} v , e_j \rangle \langle e_i, \nabla  _{e_j} w \rangle\, dx \\
& \quad + 2 \int_M  \sum_{i=1}^m \langle  \nabla _{e_i} v, \nabla _{e_i} w\rangle \, dx  - \int _M \langle \ric ^M (v), w \rangle\, dx.
\end{aligned}
\end{equation}

In view of \eqref{2.50} and the definition of curvature tensor $R$, we have 
\begin{equation}\label{8.3}
\begin{aligned}
&\div(\nabla_vv)-v(\div v)-\langle \ric ^M(v),v\rangle\\
&=\div(\sum_{i=1}^m\langle \nabla_vv,e_i\rangle e_i)-\sum_{i=1}^m\langle e_i,\nabla_v\nabla_{e_i}v\rangle-\sum_{i=1}^m\langle R(e_i,v)v,e_i\rangle\\
&=\div(\sum_{i=1}^m\langle \nabla_vv,e_i\rangle e_i)-\sum_{i,j=1}^m\langle v,e_j\rangle\langle e_i,\nabla_{e_j}\nabla_{e_i}v+R(e_i,e_j)v\rangle
\end{aligned}
\end{equation}
\[
\begin{aligned}
&=\sum_{i=1}^m e_i \langle e_i,\nabla_vv\rangle-\sum_{i,j=1}^m\langle v,e_j\rangle\langle e_i,\nabla_{e_i}\nabla_{e_j}v\rangle\\
&=\sum_{i,j=1}^m e_i (\langle v, e_j \rangle \langle  e_i,\nabla_{e_j}v\rangle)-\sum_{i,j=1}^m\langle v,e_j\rangle(e_i \langle e_i,\nabla_{e_j}v\rangle)\\
&=\sum_{i,j=1}^m ( e_i \langle v, e_j \rangle) \langle  e_i,\nabla_{e_j}v\rangle\\
&=\sum_{i,j=1}^m\langle \nabla_{e_i}v,e_j\rangle\langle e_i,\nabla_{e_j} v\rangle,
\end{aligned}
\]
Integrating \eqref{8.3} over $M$, and applying the Divergence Theorem we have
\begin{align}\int_M\sum_{i,j=1}^m\langle \nabla_{e_i}v,e_j\rangle\langle e_i,\nabla_{e_j} v\rangle \, dx=\int_M -v(\div v)-\langle \ric^M(v),v\rangle \ dx\, .\label{8.4}\end{align}

Since
$L_vg(X,Y)=\langle \n_X v,Y\rangle+\langle X,\n_Y v\rangle$,
where $g$ is the Riemannian metric on $M$ and $L_vg$ is the Lie derivative of
$g$ in the direction of $v\, ,$
\begin{equation}\label{8.5}
\begin{aligned}
\frac{1}{2}|L_vg|^2& = \frac 12 \sum_{i,j=1}^m (\langle\nabla_{e_i} v,e_j\rangle+\langle {e_i},\nabla_{e_j}v\rangle) (\langle\nabla_{e_i} v,e_j\rangle+\langle {e_i},\nabla_{e_j}v\rangle)\\
& = \sum_{i=1}^m \langle \nabla _{e_i} v , \nabla_{e_i} v\rangle+\sum_{i,j=1}^m\langle\nabla_{e_i} v, e_j\rangle\langle\nabla_{e_j} v, e_i\rangle.
\end{aligned}
\end{equation}

Following Eells-Lemaire (\cite {20}), let $\flat$ denote
the bundle isomorphism from $T(M)$ to $T^*(M)$ and $\sharp$ denote its inverse,
 where the metric on $M$ is used. We define the differential $\ov d$,
codifferential $\ov d^*$  and  Laplacian $-\ov \D$ of $v\in \Gamma(T(M))$ as
\begin{eqnarray}\ov d v=(dv^{\flat})^{\sharp},\quad \ov d^*v=d^*v^{\flat},\quad
-\ov \D v=(-\D v^{\flat})^{\sharp},\label{8.6}\end{eqnarray}
where $-\D$ denotes the de Rham-Hodge operator on 1-forms.

\noindent
Then
\begin{equation}
\begin{aligned}v(\div v) & = v \sum _{i=1}^m \langle \n _{e_i} v, e_i \rangle = v \sum _{i=1}^m (\n _{e_i} v^{\flat})(e_i) = v(- d^{\ast} v^{\flat})\\
& = \big (d (- d^{\ast} v^{\flat})\big )(v) = - \langle d d^{\ast} v^{\flat}, v^{\flat}\rangle\quad \text{which}\quad \text{implies} \\
\int _M v(\div v)\, dx  & = - \int _M \langle d^{\ast} v^{\flat}, d^{\ast} v^{\flat}\rangle\, dx  = - \int _M \langle \ov d^*v, \ov d^*v\rangle\, dx.
\end{aligned}\label{8.7}
\end{equation}

\noindent
\begin{theorem}
\begin{eqnarray}
\qquad \vp_{_\Phi}(v)&=& \int_M  \big (2 \sum_{i=1}^m \langle  \nabla _{e_i} v, \nabla _{e_i} v\rangle \big ) - v (\div v) - 2  \langle \ric^M (v), v \rangle\, dx \label{8.8}\\
&=& \int_M |L_vg|^2+v(\div v)\ dx\, \label{8.9}\\
&=&\int_M  \big (2 \sum_{i=1}^m \langle  \nabla _{e_i} v, \nabla _{e_i} v\rangle \big ) + \langle \ov d^{\ast} v, \ov d^{\ast} v\rangle - 2  \langle \ric ^M(v), v \rangle\, dx \label{8.10}\\
&=&\int_M -2 \langle\overline{\triangle}v,v\rangle-v(\div v)-4\langle \ric ^M(v), v \rangle\, dx.\label{8.11}
\end{eqnarray}
\end{theorem}

\begin{proof}
Combining \e{8.1},\e{8.2} and \e{8.4}, we obtain \e{8.8}. \e{8.9} follows from \e{8.8}, \e{8.5} and \e{8.4}. Substituting \e{8.7} into \e{8.8}, we have \e{8.10}.
By using the Weitzenb\"{o}ck formula for the Laplacian on 1-forms, we have
\begin{equation}
-\overline{\triangle}v=-\sum_{i=1}^m \nabla _{e_i} \nabla _{e_i} v +\text{Ric}^M(v).\label{8.12}
\end{equation}
\begin{align}\int_M  2 \sum_{i=1}^m \langle  \nabla _{e_i} v, \nabla _{e_i} v\rangle \, dx = \int_M -2 \langle\sum_{i=1}^m \nabla _{e_i} \nabla _{e_i} v,v\rangle \, dx .\label{8.13}
\end{align}
Substituting \e{8.13} into \e{8.8} and applying \e{8.12}, we have \e{8.11}.
\end{proof}

Since the eigenspace of $-\triangle$ is finite dimensional, based on Proposition \ref{P:8.1}, if $M$ is an Einstein manifold, we identify the following:
\begin{defn}
The $\Phi$-index of $\rm {Id}:M\rightarrow M\, ,$ denoted by $\Phi$-$\text{index} (\text {Id}_M)$ is the dimension of the largest subspace of $\Gamma(\text {Id}^{-1}TM)$ on which $H^{\Phi}_{\text {Id}}$ is negative-definite
and $\Phi$-nullity of $\text {Id}:M\rightarrow M\, ,$ denoted by $\Phi$-$\text{nullity} (\text {Id}_M)$
is the dimension of the subspace of $\Gamma(\text {Id}^{-1}TM)$ formed by the elements $v$ such that $H^{E_{\Phi}}_{\text {Id}}(v,\omega)=0$ for all $\omega\in \Gamma(\text {Id}^{-1}TM)$.
The identity map $\text {Id}$ is said to be $\Phi$-stable
if $\Phi$-index of $\text {Id}$ is zero.
\end{defn}

We prove that the identity map on any compact manifold $M$ with $\operatorname{Ric}^M\leq 0$ is $p$-stable and $p$-$\operatorname{nullity} (\operatorname{Id}_M) \leq m$ for $p > 1\, .$ (cf. \cite {13} Theorem 5.14, p.656). In parallel, 

\begin{theorem}
Let $M$ be a compact manifold with $\operatorname{Ric}^M\leq 0$. Then $\operatorname{Id}$ is $\Phi$-stable and $\Phi$-$\operatorname{nullity} (\operatorname{Id}_M) \leq m$.
\end{theorem}
\begin{proof}
The stability follows immediately from \e{8.10}. If $\text{Ric}^M\leq 0$ and $I^\Phi_{\operatorname{Id}}(v,v)=0$, then by \e{8.10} $v$ is parallel. This completes the proof.
\end{proof}

\begin{theorem}
The identity map on very compact manifold $M$ which supports a nonisometric, conformal vector field $v$ is $\Phi$-unstable for $m>4$, and is $p$-unstable for $1 < p < m\, .$ 
\end{theorem}
\begin{proof}
Since the vector field $v$ on $M$ is conformal if and only if $L_vg=-\frac{2}{m}(\div v)g$, $|L_vg|^2=\frac{4}{m}(\div v)^2$ (cf. e.g. \cite {13}). It follows from the equation \e{8.7} that
$\int _M v(\div v) \, dx = -  \int _M \langle d^{\ast} v^{\flat}, d^{\ast} v^{\flat}\rangle\, dx = - \int _M (\div v)^2\, dx  $. Substituting this into \e{8.9}, we have
\begin{eqnarray}
\vp_{\Phi}(v)=\int_M\frac{4-m}{m}(\div v)^2\, dx\leq 0.\nonumber
\end{eqnarray}
If $v$ is nonisometric conformal, we have $\div v\neq0$ and $H^\Phi_{\text {Id}}(v,v)<0$. The last assertion for $1 < p < m$ is a result in \cite [p.656]{13} .
\end{proof}
\begin{cor}\label{C:8.1}
Every compact homogenous space of dimension $m>4$ can be given a metric for which $\text {Id}_M$ is $\Phi$-unstable.
\end{cor}
\begin{proof}
Given a one-parameter group of isometries $\{u_t\}$ for the metric on $M$, one can construct a new metric under which $\{u_t\}$ is only conformal (cf. \cite {KoN}, p. 310).
\end{proof}
\begin{cor}\label{C:8.2}
Every compact manifold $M$ which possesses at least one parameter group of isometries admits a metric for which $\text {Id}:M\rightarrow M$ is $\Phi$-unstable for $m>4$.
\end{cor}
\begin{rem} Analogous to Corollaries \ref{C:8.1} and \ref{C:8.2}, every compact homogenous space of dimension $m>p$ can be given a metric for which $\text {Id}_M$ is $p$-unstable, and every compact manifold $M$ which possesses at least one parameter group of isometries admits a metric for which $\text {Id}:M\rightarrow M$ is $p$-unstable for $p<m$ (cf. \cite [5.11 and 5.12, p. 653] {13}).
\end{rem}
The scalar curvature of $M$, denoted by $\text{Scal}^M$, is the trace of the Ricci
curvature operator on M.

\begin{prop}\label{P:8.1}
Let $M$ be a compact manifold with $\langle \text{Ric}^M (v), v \rangle _M =\frac {1}{m}\text{Scal}^M\, ,$ for every unit vector $v$ at every point of $M\, .$ Then $\text {Id}_M$ is $\Phi$-unstable if and only if $\lambda_1<\frac{4}{3m}\text{Scal} ^{\, M}$.
\end{prop}
In contrast, the identity map on a compact Einstein manifold $M$ is $p$-unstable if and only if  $\lambda_1 < \frac {2}{m+p-2} \text{Scal} ^{\, M}$ (cf. \cite [Theorem 5.1, p.654] {13}).
\begin{proof}
Proceed as in \cite {13}, if $(f,\lambda)$ is an eigenpair with $\lambda<\frac{4}{3m}\text{Scal} ^{\, M}$, then by \eqref{8.6}, $v=(df)^\sharp$ is a vector field satisfying $$-\overline{\triangle}v=\bigg (-\D \big ((df)^\sharp\big )^\flat\bigg )^\sharp=(d(-\D )f)^\sharp=\lambda v$$ and via \e{8.7}, we have
\begin{align} v(\div v) = \langle \D df, df\rangle = - \langle \lambda  df, df\rangle = - \lambda \langle df, df \rangle = - \lambda \langle (df)^\sharp, (df)^\sharp \rangle. \label{8.17}\end{align}
Substituting \e{8.17} into \e{8.11} yields
\begin{equation}
\vp_{\Phi}\big ((df)^\sharp\big )=3\int_M(\lambda-\frac{4}{3m}\text{Scal}^{\, M}) |(df)^\sharp|^2\, dx < 0,\label{8.18}
\end{equation}
that is, $\text {Id}:M\rightarrow M$ is $\Phi$-unstable.

\noindent
Conversely, if $v\in \Gamma(TM)$ satisfies $-\ov\D v=\la v$ with
$\la<\frac{4}{3m}\text{Scal}^{\, M}$, then its Hodge decomposition $v^\flat=df+\sa$ with
$d^*\sa=0$ $\big ($hence $\div(\sa^\sharp)=\sum _{i=1}^m (\nabla _{e_i} \sa) (e_i) = - d^* \sa = 0 \big )$.  Since $df$ and $\sa$ lie in orthogonal
subspaces invariant by $\D$, $-\D df=\la df$, 
$-\D\sa=\la\sa$, and hence $-\ov\D (\sa^\sharp)=\la \sa^\sharp$.  We claim that $\sa\equiv 0$, for otherwise, via \e{8.7} and \e{8.11}
\begin{eqnarray}
\vp_{\Phi}(\sigma^\sharp)=\int_M 2(\la-\frac{2}{m} \text{Scal}^M) |\sigma^\sharp|^2\, dx <0.
\end{eqnarray}
On the other hand, via \e{8.9} and \e{8.7} we have
\begin{eqnarray}
\vp_{\Phi}(\sigma^\sharp)=\int_M|L_{\sigma^\sharp}g|^2dv_g\geq 0,
\end{eqnarray}
which is a
contradiction.

We conclude that $v=(df)^\sharp$; Indeed, by \e{8.18} $\vp_{\Phi} \big ((df)^\sharp \big ) < 0$. Since $d(-\D f)=d(\la f)$,  by adding a suitable constant, we have
$-\D f=\la f$.  In conclusion the set on which $\varphi$ is negative is generated by
the differentials of eigenfunctions $\la$ of $-\D$ with positive eigenvalues
$\la<\frac{4}{3m}\text{Scal}^{\, M}\, .$
\end{proof}

Recall

\noindent
\begin{prop}\label{P:8.2} (c.f. (2.8) in \cite{Smi}, p.233)\ \ If $M$ is an
Einstein $m$-manifold with $\operatorname{Ricci}\, \operatorname{tensor} = c\, \langle \, , \, \rangle _M$
for some $c\in \mathbb R$, then $c=\frac 1m \operatorname{Scal}^M$ \end{prop}

Let $\lambda(r)$=$\sharp\{$eigenvalues $\lambda$ of $-\triangle:0<\lambda<r\}$ and $m(r)$ is the multiplicity of $r$ (with $m(0)$ defined to be 0). Then we have the following result:
\begin{theorem}
Let $M$ be a compact oriented Einstein manifold with $\langle \text{Ric}^M (v), v \rangle _M =c\, $ for every unit vector $v$ at every point of $M\, ,$ where $c \in \mathbb R$ is a constant. Then we have $(a)$ $\Phi$-$\operatorname{index} (\operatorname{Id}_M) = \lambda(\frac{4c}{3})$,\\
$(b)$ $\Phi$-$\operatorname{nullity} (\operatorname{Id}_M) = \operatorname{dim} (\underline{i})+m(\frac{4c}{3})$,\\
$(c)$$($\cite {13}$)$ $p$-$\operatorname{index} (\operatorname{Id}_M) = \lambda(\frac{2mc}{m+p-2})$,\\
$(d)$$($\cite {13}$)$ $p$-$\operatorname{nullity} (\operatorname{Id}_M) = \operatorname{dim} (\underline{i})+m(\frac{2mc}{m+p-2})$,\\
where $\underline{i}$ denotes the algebra of infinitesimal isometries, i.e. of vector fields $v$ satisfying $L_vg=0$.
\end{theorem}
\begin{proof}
$(a)$ From Proposition \ref{P:8.1} and $c=\frac{\operatorname{Scal}^M}{m}$, 
\begin{eqnarray}\Phi\operatorname{-index}(\operatorname{Id}_M) = \lambda(\frac{4\operatorname{Sca}l^M}{3m})=\lambda(\frac{4c}{3}).\nonumber\end{eqnarray}
$(b)$ Since $H^{E_{\Phi}}_{\text {Id}}$ is a bilinear two form, we only need to find the dimension of the space $\{v: H^{E_{\Phi}}_{\text {Id}}(v,v)=0\}$.
From \e{8.9}, \e{8.7} and Proposition \ref{P:8.1}, we have
\begin{eqnarray}
\operatorname{dim} \{v: H^{E_{\Phi}}_{\text {Id}}(v,v)=0\}=\operatorname{dim} (\underline{i})+m(\frac{4c}{3}).
\end{eqnarray}
\end{proof}

\section{Compact  irreducible $\Phi$-SSU homogeneous space}
\begin{thm}\label{T:9.1}
Let $M=G/H$ be a compact irreducible Riemannian homogeneous space with first eigenvalues $\lambda_1$ and scalar curvature $\text{Scal} ^{\, M}$. The following four statements $($A$)$ through $($D$)$ are equivalent.\smallskip

\noindent
$($A$)$ $\lambda_1<\frac{4}{3m}\text{Scal} ^{\, M}.$\\
$($B$)$ $M$ is $\Phi$-U; i.e., The identity map on $M$ is $\Phi$-unstable.\\
$($C$)$ $M$ is $\Phi$-SU.\\
$($D$)$ $M$ is $\Phi$-SSU.
\end{thm}
\begin{proof}
Assume $\lambda_1<\frac{4}{3m}\text{Scal} ^{\, M}\, .$ Then $\lambda_1<\frac{2}{3}\, .$ For the Cartan-Killing metric $g_0$ on $M$ has Scalar curvature $\text{Scal} ^{\, M} = \frac {m}{2}\, .$ Let $M_1$ denote $M$ with the metric $g_1 = \frac {\lambda_1}{m} g_0\, .$ Then $\text{Scal} ^{\, M_1} = \frac {m^2}{2\lambda_1}\, ,$ $\text{Ric} ^{\, M_1} \equiv \frac {m}{2\lambda_1}\, ,$ and by Takahashi Theorem, there exists an isometric minimal immersion of $(M_1, g_1)$ in the unit sphere (cf. \cite {Ta}); so
\[
\begin{aligned}
\frac{3m} {4}-\ric^{M_1} & = \frac{3m} {4}- \frac {m}{2\lambda_1} \\
& = \frac {m}{2} (\frac{3} {2}- \frac {1}{\lambda_1})\\
& < 0.
\end{aligned}
\]
It follows from Corollary \ref{C:5.4} that $M_1 = (M, \frac {\lambda_1}{m} g_0)\, $ is $\Phi$-SSU. Since the metric change from $\frac {\lambda_1}{m} g_0$ to $g_0$ does not change the sign of \e{1.5} and hence by Definition \ref{D:1.3} preserves $(M, g_0)$ to be $\Phi$-SSU.  We  conclude that $($A$)$ $\Rightarrow$  $($D$)\, .$ It follows from Theorem $1.1. (a), (b), (c), (d)$ that $($D$)$ $\Rightarrow$ $($C$)\, .$ That $($C$)$ $\Rightarrow$ $($B$)$ is obvious.

Since a compact irreducible homogeous space M is an Einstein manifold, $($B$)$ $\Rightarrow$ $($A$)$ follows from Proposition \ref{P:8.1} .
\end{proof}

{\bf{Acknowledgements}}: This work was written while the first author visited Department of Mathematics of the University of Oklahoma in USA. He
would like to express his sincere thanks to Professor Shihshu Walter Wei for his help, hospitality and support.

\end{document}